\documentclass[a4paper, 10pt]{article}

\usepackage[english]{babel}
\usepackage[utf8]{inputenc}
\usepackage[T1]{fontenc}
\usepackage{amsmath,amssymb,amsfonts,amsopn,amsthm}
\usepackage{mathtools}
\usepackage{booktabs}

\usepackage{dsfont}
\usepackage{algorithm}
\usepackage{algpseudocode}
\usepackage{etoolbox}
\makeatletter
\patchcmd{\ALG@doentity}{\addvspace{0.0em}}{\addvspace{4em}}{}{}
\makeatother

\usepackage{amsfonts}

\usepackage[dvipsnames]{xcolor}
\usepackage{graphicx,epstopdf} 
\usepackage[labelformat=simple]{subcaption}
\usepackage[format=plain,
            textfont=it,
            size=footnotesize,labelsep=period]{caption}

\usepackage{tikz}
\usetikzlibrary{spy}
\usepackage{pgfplots}
\usepackage{siunitx,array,booktabs}
\usepackage{ifpdf} 
\usetikzlibrary{patterns}
\pgfdeclarelayer{foreground layer} 
\pgfsetlayers{main,foreground layer}
\usepackage{commath}
\usepackage{bm}
\usepackage{enumitem}

\usepackage[colorlinks = true,linkcolor = black,urlcolor  = black,
            citecolor = black, anchorcolor = black]{hyperref}
\usepackage[capitalise,noabbrev]{cleveref}
\crefname{equation}{equation}{equations}
\Crefname{equation}{Equation}{Equations}

\usepackage{authblk}
\usepackage{csquotes}

\pgfplotsset{compat=1.11}

\numberwithin{equation}{section}
\newtheorem{theorem}{Theorem}[section]

\newtheorem{lemma}[theorem]{Lemma}

\newtheorem{definition}[theorem]{Definition}

\theoremstyle{remark}
\newtheorem{remark}[theorem]{Remark}

\crefname{assumption}{Assumption}{Assumptions}
\crefname{remark}{Remark}{Remarks}
\crefname{example}{Example}{Examples}

\usepackage{calligra}
\usepackage{csvsimple}

\newcommand{\Nb}{\mathbb{N}}
\newcommand{\Rb}{\mathbb{R}}

\newcommand{\Cb}{\mathbb{C}}

\newcommand{\Ren}{\mathbb{R}^{n}}

\DeclareMathOperator{\arccosh}{arccosh}

\newcommand{\bigo}[1]{\mathcal{O}(#1)}
\newcommand{\dt}{\partial_t}

\newcommand{\oz}{\omega_0}
\newcommand{\ou}{\omega_1}
\newcommand{\vz}{\upsilon_0}
\newcommand{\vu}{\upsilon_1}
\newcommand{\fs}{f_S}
\newcommand{\ff}{f_F}
\newcommand{\rhos}{\rho_S}
\newcommand{\rhof}{\rho_F}

\newcommand{\costs}{c_S}
\newcommand{\costf}{c_F}
\newcommand{\bx}{\bm{x}}

\newcommand{\bfe}{\overline{f}_\eta}
\newcommand{\bfei}{\overline{f}_{\eta,1}}

\renewcommand{\Pi}{P} 
\newcommand{\Pe}{R} 
\newcommand{\Ps}{\Phi} 
\newcommand{\OF}{\Omega_F}

\title{Second order explicit stabilized multirate method for stiff differential equations with error control}

\author{Mathieu Benninghoff\textsuperscript{a} and Gilles Vilmart\textsuperscript{b}}

\date{\today}

\begin{document}

\footnotetext[1]{Universit\'e de Gen\`eve, Section de math\'ematiques, CP 64, 1211 Gen\`eve 4, Switzerland, {Mathieu.Benninghoff@unige.ch}}
\footnotetext[2]{Universit\'e de Gen\`eve, Section de math\'ematiques, CP 64, 1211 Gen\`eve 4, Switzerland, Gilles.Vilmart@unige.ch}

\maketitle
\begin{abstract}
Explicit stabilized methods are highly efficient time integrators for large and stiff systems of ordinary differential equations especially when applied to semi-discrete parabolic problems. However, when local spatial mesh refinement is introduced, their efficiency decreases, since the stiffness is driven by only the smallest mesh element. A natural approach is to split the system into fast stiff and slower mildly stiff components. In this context, [A. Abdulle, M.J. Grote and G. Rosilho de Souza 2022] proposed the order one multirate explicit stabilized method (mRKC). We extend their approach to second order and introduce the new multirate ROCK2 method (mROCK2), which achieves high precision and allows a step-size strategy with error control. Numerical methods including the heat equation with local spatial mesh refinements confirm the accuracy and efficiency of the scheme.
\end{abstract}



\section{Introduction}
Semi-discrete parabolic problems yield large systems of ordinary differential equations (ODEs) where the stiffness is strongly influenced by the spatial mesh refinement. In fact, considering a finite element discretization, the stiffness depends only on the diameter of the smallest element. Consequently, when the spatial discretization involves local  refinements, the stiffness of the system and the stability of numerical integrators is determined by only a few elements. A natural strategy is to split the system into two parts: an expensive but only mildly stiff component, associated with relatively slow (S) time scales, and a cheap but severely stiff component, associated with fast (F) time scales. This leads us to consider the system:
\begin{equation}\label{eq:ode}
\dot{y} =f(y)\vcentcolon = \ff(y)+\fs(y), \qquad\qquad
y(0)=y_0,
\end{equation}
where $\dot y(t)=\frac{dy}{dt}$ and the vector field $f:\Ren \rightarrow \Ren$ can be spitted into $\fs$, the expensive but only mildly stiff part and $\ff$, the cheap but severely stiff part. This decomposition also arises naturally in other contexts, such as chemical reaction kinetics and the modeling of electrical circuits.

For stiff problems on high-dimensional systems, standard explicit Runge-Kutta methods are limited by severe time-step restrictions, while implicit methods are computationally costly for large systems. Explicit stabilized Runge–Kutta schemes, such as RKC based on Chebyshev polynomials~\cite{HoS80, VaS80, VHS90, SSV98} (see also the presentation in \cite{HW96,Abd13c}), and ROCK methods which rely on orthogonal polynomials with quasi-optimal stability~\cite{Abd02,AbM01}, overcome this limitation by enlarging stability regions along the negative real axis, making them efficient for uniformly refined meshes, see for instance the explicit stabilized method PIROCK \cite{AbV13} buit upon ROCK2 and applied recently in astrophysics in the context of solar atmosphere simulations \cite{WVMP24}.
However, their performance deteriorates under local mesh refinement, where stiffness depends only on a few fine elements.

For such problems, it is desirable to construct methods that combine a multirate strategy with the favorable stability properties of explicit stabilized schemes. A partitioned approach has been proposed in the literature \cite{Alm23, Zbi11} and more recently in \cite{TaH25}, but these methods are adapted to problems where the slow component is assumed non-stiff, which is not the case in the context of local mesh refinement. 
To overcome this limitation, a method combining multirate strategy and stabilized integrators was introduced by A.~Abdulle, G.R.~de Souza and M.~J.~Grote \cite{AbG22}: the multirate RKC scheme (mRKC)\footnote{In the literature RKC usually refers to the second order Runge-Kutta Chebyshev methods~\cite{VaS80, VHS90, SSV98}. In this paper, it shall refer to the order one version with optimal stability domain size used for the design of mRKC \cite{AbG22}.}. This scheme applies explicit stabilized Runge--Kutta techniques to an averaged force \( f_\eta \), whose stiffness is independent of \( f_F \). A micro-step procedure with step size~$\eta$ is then employed to approximate the averaged force while stabilizing the fast component~$f_F$, whereas the macro-step~$\tau$ is used to integrate the full system and stabilize the slow component~$f_S$. These techniques have been developed in the context of first-order methods. 
Extending such stabilized multirate methods to high order is however challenging. Indeed, an earlier attempt in \cite{Abd20} is an interpolation-based approach that employs two different explicit stabilized methods, but which suffers from undesired  order reduction phenomena and instabilities.

In this paper, we extend this approach \cite{Abd20,AbG22} by introducing mROCK2, a second-order explicit stabilized multirate method based on the same averaged-force concept. Analogously to the first order mRKC, mROCK2 does not require any scale separation or interpolation between stages, but its improved order two of convergence yields  not only a better accuracy, but also it allows for an adaptive step-size strategy with error control which is essential of an efficient time integration of large dimensional dissipative systems.

To construct this method, this paper is organized as follows. Section~2 builds on the approach introduced in~\cite{AbG22}, focusing on a modified equation whose spectrum is bounded by the slow component of~\eqref{eq:ode}. While the previous work relies on a modified equation providing a first-order approximation of the original problem, we extend this approach by introducing a new modified equation that achieves a second-order approximation. In Section~3, we present the stabilized methods used to approximate the fast and slow components of the system based on the modified equation. Section~4 is the presentation and the convergence and stability analysis of the new second-order multirate stabilized scheme, the mROCK2 method. 
Finally, in Section~5, we demonstrate the effectiveness and efficiency of our approach through applications to stiff ordinary and partial differential equations (PDEs).

\section{Averaged force and modified equation}
The goal of this section is to construct a multirate strategy based on explicit stabilized methods. Inspired by \cite{AbG22}, we replace the force function $f$ with an averaged force whose stiffness depends only on the spectral radius of the slow component. The micro-macro strategy consists of first approximating with a micro method with step size $\eta$ this averaged force, and then applying a macro method to the corresponding modified equation with a macro-step size $\tau$.
We now explain how the averaged force proposed in \cite{AbG22} can be modified to to attain second-order accuracy and we analyze its stability properties.

\subsection{Averaged force and limitations for acheiving second-order}
The main idea in~\cite{AbG22} is to introduce a modified equation
\begin{equation}\label{def:fe}
\dot{y}_{\eta,1} = f_{\eta,1}(y_{\eta}), 
\quad y_{\eta,1}(0) = y_0, 
\end{equation}
where the original force function $f$ in~(1.1) is replaced by an averaged force $f_{\eta,1}$ that depends on a parameter $\eta >0$ and defined by
\begin{equation}
f_{\eta,1}(y) = \frac{1}{\eta}\bigl(u(\eta) - y\bigr)
\label{eq:deff1}
\end{equation}
where the auxiliary solution $u : [0,\eta] \to \mathbb{R}^n$ is given by
\begin{equation}
\dot{u} = f_F(u) + f_S(y), 
\quad u(0) = y. 
\label{eq:defu1}
\end{equation}

This modification results in a reduced spectral radius. Specifically, the spectrum of the Jacobian $\partial f_{\eta,1} / \partial y$, denoted by $\rho_{\eta,1}$, is smaller than that of the Jacobian $\partial f_S / \partial y$, denoted by $\rho_S$, for $\eta > \tfrac{2}{\rho_S}$. Thus, the stiffness is governed only by the slow component $f_S$, which makes the integration of the modified system significantly cheaper.  
However, the averaged vector field $f_{\eta,1}$ does not yield straightforwardly a second-order method. Indeed, if we perform a Taylor expansion of $y_{\eta,1}$ around $y_0=y(0)$, we obtain 
$
y_{\eta,1} - y_0 = \tau f_{\eta,1}(y_0) + \frac{\tau^2}{2} f_{\eta,1}'(y_0) f(y) + \mathcal{O}(\tau^3) $ and $
f_{\eta,1}(y_0)  = f(y)+ \tfrac{\eta}{2} f_F'(y_0) f(y_0) + \mathcal{O}(\eta^2).
$ Then, it follows that  
$
y_{\eta,1} - y(\eta) 
= \tfrac{\tau \eta}{2} f_F'(y_0) f(y_0) 
   + \tfrac{\tau^2}{2} f_{\eta,1}'(y_0) f(y_0) -\tfrac{\tau^2}{2} f'(y_0) f(y_0)
   + \mathcal{O}(\tau \eta^2 + \tau^2 \eta+\tau^3).
$
 Hence, when using this modified equation, we obtain only a first-order approximation. To construct a second-order scheme, we need to modifiy the definition of of the averaged force \eqref{eq:deff1} and introduce a new modification $f_{\eta,2}$, which satisfies
\begin{equation}
f_{\eta,2}(y) = f(y) + \mathcal{O}(\eta^2).
\end{equation}
In the next section, we detail the construction of such a \textit{second-order averaged force}.
 
\subsection{The second order averaged force} \label{sec:f2}
Analogously to the first-order case, the original force function \( f \) from \cref{eq:ode} is replaced by an \emph{averaged force} \( f_{\eta,2} \). The resulting equation becomes
\begin{equation}\label{eq:odemod2}
\dot y_{\eta,2}=f_{\eta,2}(y_{\eta}), \qquad\qquad
y_{\eta}(0)=y_0,  
\end{equation}
\begin{definition}
	For $\eta>0$, the second-order averaged force $f_{\eta,2}:\Ren\rightarrow\Ren$ is defined as
	\begin{equation}\label{def:fe2}
		f_{\eta,2}(y)=\frac{1}{\eta}(v(\eta)-y),
	\end{equation}
	where  the auxiliary solution $v:[0,\eta]\rightarrow \Ren$ is defined by solving successively two fast ODE systems. Considering first $u:[0,\eta]\rightarrow \Ren$ given by:
	\begin{align}\label{eq:defu2}
	\dot{u}&=\ff(u)+\fs(y), & 
	u(0)=y,
	\end{align}
    we obtain as an intermediate step $f_{\eta,1}(y)=\frac{1}{\eta}(u(\eta)-y)$, the classical averaged-force as defined in \eqref{eq:defu1}. Then we consider the differential equation for $v(t)$ given by:
    \begin{align}\label{eq:defv}
	\dot{v}&=\ff\left(v-\frac{\eta}{2}f_{\eta,1}(y)\right)+\fs(y), & 
	v(0)=y.
	\end{align}
    For $\eta=0$, let $f_{0,2}=f$.
    \end{definition}
This second-order averaged force also provides a spectral compression, Theorem \ref{secondordertheorem} (\cref{sec:stab_modeq}) establishes that the spectral radius satisfies $\rho_{\eta,2} < 1.5\,\rho_s$. In exchange for increased accuracy, we incur an additional computational cost: the evaluation of $f_F$ twice per time step, which is acceptable since $f_F$ is the cheaper part of the force. Furthermore, although spectral reduction is slightly less effective than with $f_{\eta,1}$, the numerical method will later improve this bound to $\bar{\rho}_{\eta,2} < 1.35\,\rho_s$ (Theorem \ref{thm:stab_mROCK2}).

\subsection{Stability analysis of the modified equation}\label{sec:stab_modeq}
We now study the stiffness of the modified equation \eqref{eq:odemod2}. We determine necessary conditions to have robust bounds on  $\rho_{\eta,2}$ of the system. 
It will only depend on $\rhos$ the spectral radius of the Jacobian of $\fs$, and hence the stiffness of the modified equation only depends on the slow components. 

Let the Jacobians of $\ff$ and $\fs$ be simultaneously diagonalizable in the same basis. Then, after linearization of the dynamics, the stability analysis of \cref{eq:ode,eq:odemod2} reduces to the scalar \textit{multirate test equation}:
\begin{equation}\label{eq:mtesteq}
\dot{y} = \lambda y+\zeta y,\qquad\qquad
y(0)=y_0,
\end{equation}
with $\lambda,\zeta\leq 0$ and $y_0\in\Rb$, which corresponds to setting $\ff(y)=\lambda y$ and $\fs(y)=\zeta y$; thus, $\rhof=|\lambda|$ and $\rhos=|\zeta|$. Since we do not assume any scale separation, we have
$\lambda \in[-\infty,0]$ where $|\lambda|$ is possibly very large. Using the variation of constant formula with $\varphi(z)=\frac{e^z-1}{z}\label{phiz}$, we first calculate $f_{\eta,1}$ obtained from \eqref{eq:defu2}:

\begin{align}\label{eq:solutesteq}
u(\eta) &= (e^{\eta\lambda}+\varphi(\eta\lambda)\eta \zeta )y, \\ \label{eq:fetesteq}
f_{\eta,1}(y) &= \varphi(\eta\lambda)(\lambda+\zeta)y
\end{align}
Then, we evaluate $v(\eta)$  \eqref{eq:defv} and $f_{\eta,2}$ defined in \eqref{def:fe2}:
\begin{align}
v(\eta)&= e^{\eta\lambda}y+\varphi(\eta\lambda)\eta \zeta y-\varphi(\eta\lambda)\frac{\eta^2\lambda}{2} f_{\eta,1}(y),\\
f_{\eta,2} (y)&=\varphi(\eta\lambda)\left(\lambda+\zeta-\frac{\eta\lambda}{2}\varphi(\eta\lambda)(\lambda+\zeta)\right)y.
\end{align}
Thus, we can rewrite the modified equation \eqref{eq:odemod2}:
\begin{equation*}
\dot y_{\eta,2}=\varphi(\eta\lambda)(\lambda+\zeta)\left(1-\frac{\eta\lambda}{2}\varphi(\eta\lambda)\right)y.
\end{equation*}
Now we can establish the stability equation for the modified equation.
\begin{theorem}\label{secondordertheorem}
Let $\zeta<0$, then $\varphi(\eta\lambda)(\lambda+\zeta)(1-\frac{\eta\lambda}{2}\varphi(\eta\lambda)) \in [\frac{3}{2}\zeta,0]$ for all $\lambda\leq 0$, if $\eta\geq 2/|\zeta|$.
\end{theorem}
\begin{proof}
By Theorem~2.7 in~\cite{AbG22}, 
\begin{equation}\label{eq:thmstab1}   
\varphi(\eta \lambda)(\lambda + \zeta) \in [-\zeta,0],
\quad \mbox{for all} \, \zeta < 0, \ \lambda < 0, \ \eta \geq 2/|\zeta|.
\end{equation}
Moreover,
\[
1 - \tfrac{\eta \lambda}{2}\varphi(\eta \lambda) 
= 1 - \tfrac{1}{2}(e^{\eta \lambda} - 1) \in [0, \tfrac{3}{2}],
\quad \text{since } e^{\eta \lambda} \in [0,1].
\]
The product of these two intervals gives the result.
\end{proof}

Thus, $\eta$ does not depend on $\lambda$ and the result holds for all $\lambda< 0$.

\section{Construction of explicit stabilized multirate methods}

In the previous section, we introduce a new second-order averaged force \( f_{\eta,2} \), extending the first-order averaged force \( f_{\eta,1} \). This construction allows the use of a multirate strategy inspired by the mRKC method~\cite{AbG22}, now adapted to a second-order context by employing \( f_{\eta,2} \). The key advantage of this multirate approach lies in its efficiency: the number of evaluations of the expensive component \( f_S \) remains independent of the stiffness of the fast component \( f_F \), leading to significant gains over classical explicit stabilized methods.

This section sets the foundation for the new multirate method by first reviewing the explicit stabilized methods  RKC \cite{HoS80} and ROCK2 methods \cite{AbM01} that we use for the design of mROCK2. In the new method, RKC is employed as the micro-method to approximate the modified equation $f_{\eta,2}$ \eqref{def:fe2}, while the ROCK2 method serves as the macro-method for integration the modified equation \eqref{eq:odemod2}.

In contrast, we recall the previous multirate RKC (mRKC) method~\cite{AbG22}, which already uses this averaging-based multirate strategy in only first-order accurate. Specifically, it uses the RKC method both to approximate the averaged force as a micro method and to solve the modified equation as a macro method.

\subsection{The explicit stabilized Runge-Kutta methods}
Stability is a fundamental property of numerical method for stiff differential equations. A method is stable if the numerical solution is bounded throughout the integration process. To study this, we use the scalar test function: 
\begin{equation}\label{testequation}
    \dot{y}=\lambda y,\qquad y(0)=y_0,
\end{equation}
with $\lambda \in \mathbb{C}$. When applying a Runge–Kutta  method (RK) with step size $\tau$ to this problem, the iteration takes the form $y_{k+1}=R(\lambda \tau)y_k$, which by induction yields $y_{k}=R(\lambda \tau)^ky_0$. The function $R(z)$ is called the stability function. The stability region is then defined as $\mathcal{S}=\{ z\in\Cb\,:\, |\Pe(z)|\leq 1\}$. Hence, $y_k$ remains bounded for all $k$ if and only if $\lambda h\in \mathcal{S}$. For implicit RK methods $R(z)$ is a rational function, while for explicit ones, $R(z)$ is a polynomial.

The explicit stabilized Runge methods \cite{Abd02,AbM01,LeM94,HoS80} (see also the survey \cite{Abd13c}) are are families of explicit Runge-Kutta methods, with a large number of internal stages built to extend the stability domain along the negative real axis. By increasing the number of internal stages $s$, these methods enlarge the portion of the negative real axis contained in the stability region ($[-\ell_s, 0] \subset \mathcal{S}$) with size that grows quadratically w.r.t. the number of internal stages $s$ as $\ell_s \sim \beta s^2$ with a constant $\beta$ independent of $s \rightarrow +\infty$. We present here the methods that we use for the design of mROCK2, the first-order RKC method \cite{HoS80} and the second-order ROCK2 method \cite{AbM01}.

\subsubsection{The RKC method}
The first-order RKC method \cite{HW96, HoS80} is not much used in practice for deterministic problems due to its low (first) order of accuracy. However, the method possesses an optimally large stability domain. It is defined by the following $m$-stage Runge-Kutta method that can be implemented as a two-term reccurence as:

\begin{equation}\label{eq:rkc}
\begin{aligned}
k_0&=y_n,\\
k_1 &= k_0+\mu_1\tau f(k_0),\\
k_j&= \nu_j k_{j-1}+\kappa_j k_{j-2}+\mu_j\tau f(k_{j-1}) \quad j=2,\ldots,m,\\
y_{n+1}&=k_m,
\end{aligned}
\end{equation}
where $\tau$ is the step size, $\mu_1 = \ou/\oz$ and
\begin{align*}\label{eq:defcoeff} 
	\mu_j&= 2\ou  b_j/b_{j-1}, &
	\nu_j&= 2\oz b_j/b_{j-1},  &
	\kappa_j&=-b_j/b_{j-2}  &
	\text{for }j&=2,\ldots,m,
\end{align*}
with $\varepsilon\geq 0$, $\oz=1+\varepsilon/m^2$, $\ou=T_m(\oz)/T_m'(\oz)$ and $b_j=1/T_j(\oz)$ for $j=0,\ldots,s$. 

When applied to the test equation \eqref{testequation} and using
\begin{align*}
T_0(x)&= 1, & T_1(x)&= x, & T_j(x)&= 2xT_{j-1}(x)-T_{j-2}(x),
\end{align*}
\cref{eq:rkc} yields $y_{n+1}=P_m(z)y_n$, where $z=\tau\lambda$ and
\begin{equation}\label{eq:stabpolrkc}
	P_m(z)=b_mT_m(\oz+\ou z)
\end{equation}
is the stability polynomial of the method. As $|T_m(x)|\leq 1$ for $x\in [-1,1]$ then $|P_m(z)|\leq 1$ for $z\in [-\ell^\varepsilon_s,0]$ and $\ell_s^\varepsilon=2\oz/\ou$. As $\beta m^2\leq \ell_m^\varepsilon$, for $\beta=2-4\varepsilon/3$, then $|z|\leq \beta m^2$ is a sufficient condition for stability (see \cite{Ver96}) and the stability domain $\mathcal{S}$ increases quadratically, with respect to the number stage $m$ on the real negative axis. The parameter $\varepsilon$ is the damping parameter which avoids for $\varepsilon>0$ points $z$ in the domain such that $P_m(z)=1$. The stability domain becomes a bit shorter with size $\ell_s^\varepsilon = \beta m^2$ where $\beta=\left(2-4/3\varepsilon\right)$ instead of the optimal length $\ell_s^0 = 2m^2$ achieved without damping $\varepsilon=0$. In Figure \ref{stabilityrkc}, we represent the stability function $P_m(z)$ for $m=10$ and different damping values $\varepsilon$. The number of stages $s$ is chosen large enough such that $\tau \rho \le \beta s^2$ where $\rho$ is the spectral radius of the Jacobian of $f$ for the equation $\dot{y}=f(y)$. In practice, we choose $s=\left\lfloor\sqrt{\frac{\tau\rho}{\beta}}\right\rfloor+1$ where $\lfloor\cdot \rfloor$ denotes the integer part function.

\begin{figure}[tbp]
    \centering
    \begin{subfigure}{\linewidth}
        \centering
        \includegraphics[width=0.9\linewidth]{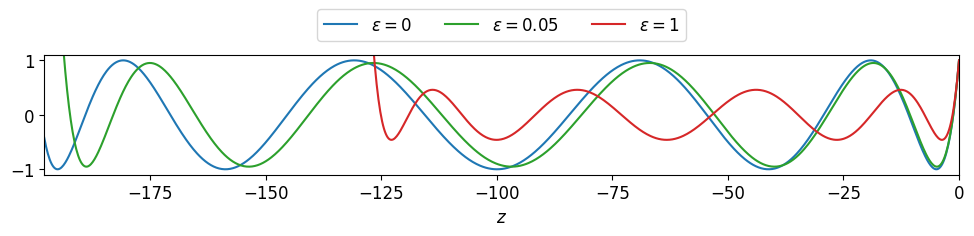}
        \vspace{-0.3em}
        \caption{Stability polynomial $P_m(z)$ of the RKC method with $m=10$ and different damping values $\varepsilon$.}
        \label{stabilityrkc}
    \end{subfigure}

    \vspace{0.3em} 

    \begin{subfigure}{\linewidth}
        \centering
        \includegraphics[width=0.9\linewidth]{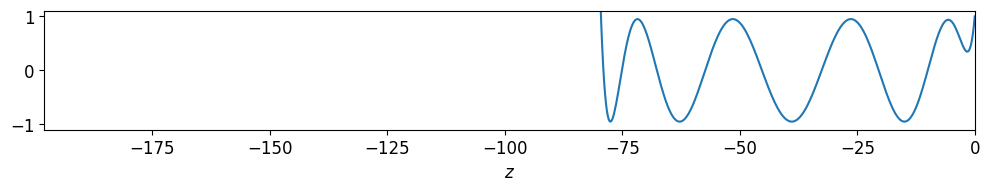}
        \vspace{-0.3em}
        \caption{Stability polynomial $R_s(z)$ of the ROCK2 method with $s=10$.}
        \label{stabilityrock2}
    \end{subfigure}
    \caption{The stability function for the RKC and the ROCK2 method.}
    
\end{figure}
\subsubsection{The ROCK2 method}
Constructing a second-order explicit stabilized is not a trivial task and various methods exist. There exist several methods as the second order RKC \cite{HW96} or the DUMKA method \cite{Lbd89}. We describe here the ROCK2 method \cite{AbM01}. The central idea behind this method was to achieve the second-order accuracy~\cite{AbM01} while preserving a three-term recurrence relation and a nearly optimal stability domain with size $0.81\,s^2$. The scheme is defined as follows:
\begin{equation}\label{def:rock2}
    \begin{aligned}
     k_0 &= y_n \\
     k_1 &= k_0 + \tau \mu_1 f(k_0) \\
     k_j &= h \mu_j f(k_{j-1}) - \nu_j k_{j-1} - \kappa_j k_{j-2}, \quad j = 2, \ldots, s - 2 \\ 
     k_{s{\scriptscriptstyle -}1} &= k_{s-2} + \tau\sigma_1 f(k_{s-2}) \\
     k_{s}^{\star} &= k_{s-1} + \tau\sigma_1 f(k_{s-1}) \\
     k_{s} &= k_s^{\star} - \tau\sigma_1(1 - \tfrac{\sigma_1}{\sigma_2^2})(f(k_{s-1}) - f(k_{s-2})),
    \end{aligned}
\end{equation}
where $\mu_j$, $\nu_j$, and $\kappa_j$ for $j=1...s-2$ are  the three-term relation of the orthogonal polynomial and $\hat{\omega}(x) =1+2\sigma_1 z+\sigma_2 z^2$.
Using the same test function as the first one, we have $R_s(z)=\hat{\omega}(z)\hat{P}_{s-2}(z)$, where \( \hat{P}_{s-2}(z) \) is a member of the orthogonal polynomial family \(\{\hat{P}_j(z)\}_{j \geq 0}\) with respect to the weight function \( \omega(x)^2 /(1 - x^2) \). Moreover, we note that 
\( \hat{\omega}(z) = \omega(1 + 2z / \ell_s) \), where \( \ell_s \) denotes 
the length of the stability domain.  
The polynomial \( P_{s-2} \) has degree \( s - 2 \), while \( \omega \) is a positive polynomial of degree 2 (depending on \( s \)).  
One constructs the polynomials \( \omega \) such that \( R_s \) satisfies  
\[
R_s(z) = 1 + z + \frac{z^2}{2} + O(z^3),
\]
together with a large domain of stability, increasing approximately as \( 0.81\, s^2 \) along the negative real axis. Figure \ref{stabilityrock2} represents  the stability function of the ROCK2 method for $s=10$. Thus the length of the negative real axis included in the stability domain is approximately $81$.

\subsection{Explicit stabilized multirate Runge-Kutta method of order one (mRKC)}
The multirate Runge–Kutta–Chebyshev method (mRKC) \cite{AbG22} is  a first-order multirate explicit stabilized method. It was the first method that approximates the averaged force \eqref{def:fe} in order to reduce its spectral radius. This scheme \cite{AbG22} is obtained by discretizing the modified equation with an $s$-stage RKC method, where $\bar{f}_{\eta,1}$, the approximation of $f_{\eta,1}$, is constructed by solving problem \eqref{eq:defu1} with one step of an $m$-stage RKC method.  Let $\tau>0$ be the step size and $\rhof,\rhos$ the spectral radii of the Jacobians of $\ff,\fs$.
Now, let the number of stages $s,m$ be the smallest integers satisfying
\begin{align}\label{eq:defsmeta1}
\tau\rhos &\leq \beta s^2, & \eta\rhof &\leq  \beta m^2, &\text{with}  && \eta &= \frac{6\tau}{\beta s^2} \frac{m^2}{m^2-1},
\end{align}
The function Get\_Stages($\Delta t,\rhos,\rhof,\epsilon)$ in Algorithm 1 returns the smallest values $s,m$ and $\eta$ that solve \ref{eq:defsmeta1}.
\vspace{-0.5em}
\begin{algorithm}[h!]
\caption{One step of the mRKC method \cite{AbG22} with damping $\epsilon = 0.05$   (as presented in \cite{RdS24}).}
\begin{algorithmic}[1]
\Function{mRKC\_Step}{$t_n, y_n, \Delta t, f_F, f_S, \rho_F, \rho_S$}
    \vspace{0.3em}
    \State $s,m,\eta \gets \textsc{Get\_Stages}(\Delta t, \rho_S, \epsilon)$ \Comment{Solve \eqref{eq:defsmeta1}}
    \vspace{0.3em}
    \State $\bar{f}_{\eta,1}(t, y) \gets \textsc{Averaged\_Force}(t, y, \eta, f_F, f_S, m)$ 
    \vspace{0.3em}
    \State $y_{n+1} \gets \textsc{RKC}(t_n, y_n, \Delta t, \bar{f}_\eta, s, \epsilon)$ \Comment{Compute $y_\eta$ from \eqref{def:fe} using RKC \eqref{eq:rkc}}
    \vspace{0.3em}
    \State \Return $y_{n+1}$
    \vspace{0.3em}
\EndFunction

\vspace{0.6em} 

\Function{Averaged\_Force}{$t, y, \eta, f_F, f_S, m$}
    \vspace{0.3em}
    \State $f_u(r, u) \gets f_F(r, u) + f_S(t, y)$ \Comment{Iterate on $f_F$ with $f_S$ frozen}
    \vspace{0.3em}
    \State $u_\eta \gets \textsc{RKC}(t, y, \eta, f_u, m, \epsilon)$ \Comment{Approximate $u(\eta)$  \eqref{eq:defu1} using RKC}
    \vspace{0.3em}
    \State \Return $\bar{f}_{\eta,1}=\frac{1}{\eta} (u_\eta - y)$ \Comment{Compute the numerical averaged force $\bar{f}_{\eta,1}$}
    \vspace{0.3em}
\EndFunction
\vspace{0.3em}
\end{algorithmic}
\end{algorithm}

In what follows, we analyze the stability of the scheme using the multirate test equation \eqref{eq:mtesteq}.
We use the stability analysis of the averaged force $f_{\eta,1}$ \eqref{eq:solutesteq} to understand the stability analysis of the multirate method. The stability analysis from the modified equation is no longer applies once  is approximated numerically, since $\varphi(z)$ is replaced by a discrete stability function. For the multirate test equation \eqref{eq:mtesteq}, the RKC approximation of $u_\eta$ is given by 
\begin{equation*}
u_\eta = \left(P_m(\eta\lambda) + \Ps_m(\eta\lambda)\eta\zeta\right)y,
\end{equation*}
where $P_m(z) = a_m T_m(\vz + \vu z)$ is the $m$-stage RKC stability polynomial, and $\Ps_m(z) = \frac{P_m(z) - 1}{z}$ for $z \neq 0$ with $\Ps_m(0)=1$. For the multirate test equation, the first order averaged force \cref{eq:fetesteq} becomes
\begin{equation}\label{eq:feta1mrkc}
\bfei(y) = \Ps_m(\eta\lambda)(\lambda + \zeta)y,
\end{equation}
and thus the full stability function of the mRKC scheme reads
\begin{equation}\label{eq:stabmrkc}
y_{n+1} = R_s\big(\tau \Ps_m(\eta\lambda)(\lambda + \zeta)\big)y_n,
\end{equation}
where $R_s(z)$ is the $s$-stage RKC stability polynomial. By Theorem 4.5 of \cite{AbG22}, the method is stable for the multirate test equation \eqref{eq:mtesteq}, namely $|R_s\big(\tau \Ps_m(\eta\lambda)(\lambda + \zeta)\big)|\le 1$ under the condition \eqref{eq:defsmeta1}. In Figure \ref{fig:stabilitymRKC}, we  display the inner (micro) \eqref{eq:feta1mrkc} and outer (macro) \eqref{eq:stabmrkc} stability polynomials for $s=m=10$ as a function of $\lambda$. We set $\tau=1$, $\eta= \frac{6\tau}{\beta s^2} \frac{m^2}{m^2-1}$, and $\zeta=\ell_s$. The first plot shows the stability region of the micro-step; the orange line indicates the boundary condition: $\tau \zeta\le\ell_S$.  The second plot show the stability of the RKC method, thus $|R_s\big(\tau \Ps_m(\eta\lambda)(\lambda + \zeta)\big)|\le 1$ under the condition \eqref{eq:defsmeta1}.
\begin{figure}[tbp]
    \centering
    
    \begin{subfigure}{0.92\linewidth}
        \centering
        \includegraphics[width=\linewidth]{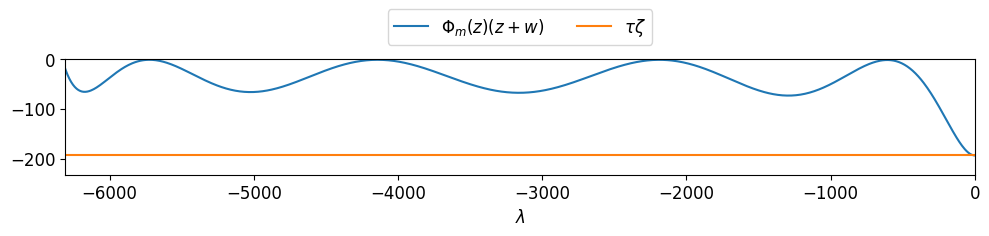}
        \vspace{-0.5cm}
        \caption{$\tau \Phi_m(z)(z+w)(1-\frac{z}{2}\alpha_m\Phi_m(z))$ for $z=\eta\lambda $, $w=\eta\zeta$ and the stability condition $\ell_s=\gamma_\varepsilon\eta\zeta$.}
        \label{subfig:haut}
    \end{subfigure}
    
    \hspace{0.3cm}\begin{subfigure}{0.9\linewidth}
        \centering
        \includegraphics[width=\linewidth]{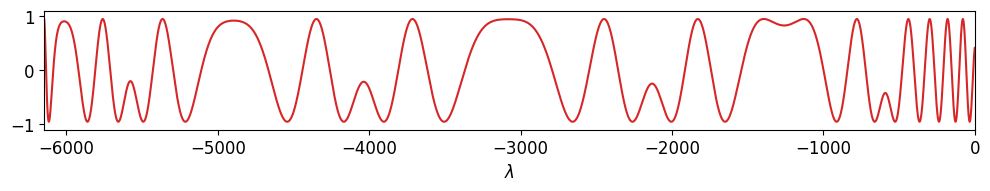}
        \caption{Stability polynomial $R_s\big(\tau \Ps_m(\eta\lambda)(\lambda + \zeta)\big)$. }
        \label{subfig:bas}
    
    \end{subfigure}

    \caption{mRKC on the multirate test equation 
        with $\tau=1$, $\eta= \tfrac{6\tau}{\beta s^2} \tfrac{m^2}{m^2-1}$, $\varepsilon=0.05$, $s=10$, 
        $m=10$, and $\zeta=\ell_s$.}
    \label{fig:stabilitymRKC}
\end{figure}

\section{New mROCK2 method: second-order multirate explicit stabilized explicit method}

This section introduces the new multirate ROCK2 method (mROCK2), which is constructed by applying an \( s \)-stage ROCK2 scheme to discretize the modified equation \eqref{eq:odemod2}, where the second-order averaged force \( f_{\eta,2} \) \eqref{eq:defu2} is approximated using two successive steps of an \( m \)-stage RKC method as introduced in Section \ref{sec:f2}.



We first introduce the mROCK2 algorithm, detailing also a variable step size selection with error control. Then we compare its efficiency gains with the ROCK2 method, and finally analyze its stability and convergence properties.

\subsection{New second-order mROCK2 method}
Let $\tau>0$ be the step size and $\rhof,\rhos$ the spectral radii of the Jacobians of $\ff,\fs$, the fast and slow component as \eqref{eq:ode}.
Now, let the number of stages $s,m$ be the smallest integers satisfying
\begin{align*}
\gamma_m\,\tau\rho_S &\leq \tilde \beta_s s^2, &
\eta\,\rho_F &\leq \beta_m m^2, &
\text{with} \quad \eta &= \frac{6\tau}{\tilde \beta_s s^2} \cdot \frac{m^2}{m^2 - 1}.
\end{align*}
Here, \(\tilde \beta_s\) denotes the stability constant of the ROCK2 method with stability domain size $\beta s^2$ (\(\tilde \beta_s \approx 0.81\) for large \(s\)), while \(\beta_m\) is the stability constant of the RKC method (\(\beta_m \approx 1.93\) for \(\varepsilon = 0.05\)). In addition, $\gamma_m$ is a constant that depends on the second derivatives of the RKC method: $\alpha_m=P_m''(0)$ and $\gamma_m=1+\alpha_m$.
The function Get\_Stages($\Delta t,\rhos,\rhof,\epsilon)$ in Algorithm 2 returns the smallest values $s,m$ and $\eta$ that solve the stability constraints \eqref{eq:defsmeta2} defined below. 
Algorithm 2 gives one step of the mROCK2 scheme.
\begin{algorithm}[h!]\label{def:mrock2}
\caption{One step of the new mROCK2 method}
\label{mROCK2def}
\begin{algorithmic}[1]
\vspace{0.3em}
\Function{mROCK2\_Step}{$t_n, y_n, \Delta t, f_F, f_S, \rho_F, \rho_S$}
\vspace{0.3em}
    \State $s,m,\eta \gets \textsc{Get\_Stages}(\Delta t, \rho_S,\rho_F,\epsilon)$ \Comment{Solve \eqref{eq:defsmeta2} }
    \vspace{0.3em}
    \State $\bar{f}_{\eta,2}(t, y) \gets \textsc{Averaged\_Force}(t, y, \eta, f_F, f_S, m)$ 
    \vspace{0.3em}
    \State $y_{n+1} \gets \textsc{ROCK2}(t_n, y_n, \Delta t, \bar{f}_{\eta,2}, s)$ \Comment{Approximate $y_\eta$ from \eqref{eq:odemod2} using ROCK2 \eqref{eq:rkc}}
    \vspace{0.3em}
    \State \Return $y_{n+1}$
    \vspace{0.3em}
\EndFunction
\vspace{0.8em}
\Function{Averaged\_Force}{$t, y, \eta, f_F, f_S, m$}
    \vspace{0.3em}
    \State $f_u( u) \gets f_F( u) + f_S( y)$ \Comment{Iterate on $f_F$ with $f_S$ frozen}
    \vspace{0.3em}
    \State $u_\eta \gets \textsc{RKC}(t, y, \eta, f_u, m, \epsilon)$ \Comment{Approximate $u(\eta)$ from  \eqref{eq:defu2} using RKC \eqref{eq:rkc}}
    \vspace{0.3em}
    \State $\bar{f}_{\eta,1}=\frac{1}{\eta} (u_\eta - y)$ 
    \vspace{0.3em}
    \State $f_v(v) \gets f_F( v-\frac{\alpha_m\eta}{2}\bar{f}_{\eta,1}) + f_S( y)$ 
    \vspace{0.3em} \Comment{As $\alpha_m=P_m''(0)$, we impose $\bar{f}_{\eta,2}(y)=f(y)+\bigo{\eta^2}$}
    \State $v _\eta \gets \textsc{RKC}(t, y, \eta, f_v, m, \epsilon)$
    \vspace{0.3em} \Comment{Approximate $v(\eta)$  from \eqref{eq:defv} using RKC a second time}
    \vspace{0.3em}
    \State \Return $\bar{f}_{\eta,2}=\frac{1}{\eta} (v_\eta - y)$ \Comment{Compute the numerical averaged force $\bar{f}_{\eta,2}$}
    \vspace{0.3em}
\EndFunction

\vspace{0.3em}
\end{algorithmic}
\end{algorithm}

To better understand how the method works, we provide in Table \ref{tabel} a comparison of the number of evaluations of the slow and fast components of the vector field for one (macro) step of the methods, as well as the size of the stability domain, with those the mRKC and ROCK2 methods. Analogously to  RKC and ROCK2, we observe that increasing the order of convergence from mRKC to mROCK2 reduces the size of the stability domain. However, this trade-off is clearly acceptable, as the improved accuracy of the second-order method compensates for the reduction in the stability domain (see Section~\ref{sec:numex}). We also included RKC and ROCK2 in the comparison, as it requires a large number of $f_S$ evaluations when $f_F$ is very stiff. In Section~\ref{sec:compcost}, we discuss in more detail the situations where the multirate method is more efficient.

\begin{table}[h]
\centering
\renewcommand{\arraystretch}{1.4} 
\begin{tabular}{|l|c|c|c|c|c|}
\hline
        & Order & $\#f_S$ & $\#f_F$ & $l_s$ & $l_m$ \\
\hline
RKC   & $1$ & $m$ & $m$       & $\approx2m^2$  & $\approx2m^2$ \\
\hline
mRKC    & $1$ & $s$ & $ms$      & $\approx2s^2$    & $\approx0.66s^2m^2$ \\
\hline
ROCK2   & $2$ & $m$ & $m$       & $\approx0.81m^2$  & $\approx0.81m^2$ \\
\hline
mROCK2  & $2$ & $s$ & $2ms$     & $\approx0.6s^2$  & $\approx0.27m^2s^2$ \\
\hline
\end{tabular}
\caption{Comparison of the number of function evaluation per time step and the sizes of the stability domain for the RKC, mRKC, ROCK2, and mROCK2 methods.}

\label{tabel}
\end{table}

\subsection{Variable time step with error control}\label{sec:step size}
A key advantage of second-order methods over first-order ones is their compatibility with variable time stepping and error control. To implement variable time steps, the error of the mROCK2 method must first be estimated. The microstep can be neglected in this estimate, while the overall error is dominated by the macrostep. Indeed, if $s\gg1$, then $\eta\ll\tau$ and the errors made when approximating $f$ by the averaged force $f_{\eta,2}$ is negligible. We observe numerically in Section \ref{sec:numex} that the numerical solutions of mROCK2 and ROCK2 are very close. Therefore, it is natural to adopt the error estimator originally developed for ROCK2 (\cite{AbM01}, Section 6.3), that we recall here for completeness:
\begin{equation*}
    \text{err}=\|k^{\star}_s-k_s\|
\end{equation*}
where $k^{\star}_s$ and $k_s$ denote, respectively, the penultimate and the final stages of the ROCK2 algorithm as defined in \eqref{def:rock2}. Finally, compared to the usual stepsize selection for non-stiff problems 
\begin{equation}\label{step size conv}
\tau_{\text{new}} = \text{fac} \cdot \tau_n 
\left(\frac{1}{\text{err}_{n+1}}\right)^{1/2} ,
\end{equation}
we shall  consider the same step size selection strategy with memory, originally introduced by Watts \cite{Wat84} and Gustafsson \cite{Gus94} to tackle the integration of stiff dynamics:
\begin{equation}\label{step size}
\tau_{\text{new}} = \text{fac} \cdot \tau_n 
\left(\frac{1}{\text{err}_{n+1}}\right)^{1/2} 
\frac{\tau_n}{\tau_{n-1}} 
\left(\frac{\text{err}_n}{\text{err}_{n+1}}\right)^{1/2}.
\end{equation}
Here we use the value $\text{fac}=0.8$ commended for stiff problems (instead of $\text{fac}=0.9$ for non-stiff Runge-Kutta methods). 
As recommended in \cite[p.~125]{HW96}, we take for $\tau_{\text{new}}$ the minimum of the step sizes proposed by \eqref{step size} and the conventional strategy \eqref{step size conv}. This conventional strategy is also applied after a step rejection which occurs if $\text{err}>\text{tol}$ where $\text{tol}$ is the brecribed accuracy. In Section~\ref{sec:numex}, we observe the efficiency of such a variable time step strategy notably for the stiff Robertson problem (see Figure~\ref{fig:step size}).

\subsection{Computation of the number of stages}
To simplify the calculation of $s$, $m$, and $\eta$, we fix the damping parameter to $\varepsilon = 0.05$, although any value $\varepsilon > 0$ could be used to increase the damping if needed.
 The general system of constraints obtained from the stability analysis in \cref{sec:staban} reads:
\begin{align}\label{eq:defsmeta}
\left(1+\alpha_m\right)\tau\rhos &\leq \tilde \beta_s s^2, & \eta\rhof &\leq  \beta_m m^2, &\text{with}  && \eta &\ge \frac{2\tau}{\alpha_m\tilde \beta_s s^2} .
\end{align}
We note that obtaining a computable system requires an estimate of $\alpha_m$. Hence, we give the following bounds.
\begin{lemma} \label{thm:optimalbound}
Let \( m \in \mathbb{N} \) and \( \varepsilon > 0 \) and define 
$\alpha_m(\varepsilon) := P_m''(0)$
where $P_m(z)$ is the stability polynomial given in \eqref{eq:stabpolrkc} of the RKC method.
Then,
\begin{equation}\label{inequalityalpham1}
\alpha_m(0) < \alpha_m(\varepsilon) < \alpha(\varepsilon),
\end{equation}
where $\alpha_m(0) = \frac{1}{3} \cdot \frac{m^2 - 1}{m^2}$ 
and 
$\alpha(\varepsilon):=\lim_{m \to \infty} \alpha_m(\varepsilon) =\frac{1}{\tanh(\sqrt{2\varepsilon})} \left( \frac{1}{\tanh(\sqrt{2\varepsilon})} - \frac{1}{\sqrt{2\varepsilon}} \right).$
%
%
\end{lemma}

The first inequality in \eqref{inequalityalpham1} follows directly from Lemma \ref{epscroi} proved in \cite{AbG22}, while the second inequality in \eqref{inequalityalpham1} is proved in Appendix. 
\begin{lemma}\label{epscroi}
\cite[Lemma A.1]{AbG22}  The function \( \alpha_m(\varepsilon) \) is an increasing function of $\varepsilon>0$.
\end{lemma}

Using \eqref{inequalityalpham1}, we obtain:
\[
1 + \alpha_m< 1 + \alpha(\varepsilon),
\quad \text{and} \quad
\frac{2\tau}{\alpha_m\tilde \beta_s s^2} \geq \frac{6\tau}{\tilde \beta_s s^2} \cdot \frac{m^2}{m^2 - 1}.
\]
For the typical value $\varepsilon = 0.05$, we estimate \(\alpha(\varepsilon) \approx 0.34 < 0.35\). Based on these bounds, we can reformulate the stability constraints as follows:
\begin{align}\label{eq:defsmeta2}
1.35\,\tau\rho_S &\leq \tilde \beta_s s^2, &
\eta\,\rho_F &\leq \beta_m m^2, &
\text{with} \quad \eta &= \frac{6\tau}{\tilde \beta_s s^2} \cdot \frac{m^2}{m^2 - 1}.
\end{align}

\begin{remark}[Relaxed stability condition] In the presence of scale separation, when $\lambda \ll \zeta$, the stability requirements can be relaxed, as observed also for mRKC in \cite{AbG22}. In this regime of scale separation, it becomes advantageous to adjust the micro--step size $\eta$ to a modified value $\bar{\eta}$, which leads to the following stability constraints for mROCK2:
\begin{align}\label{eq:defsmeta3}
1.35\,\tau\rho_S &\leq \tilde \beta_s s^2, &
\bar \eta\,\rho_F &\leq \beta_m m^2, &
\text{with} \quad \bar \eta &= \frac{2.8}{\tilde \beta_s s^2} 
.
\end{align}
For more details, see Remark \ref{remarkscale}. The value of $\bar \eta$ is clearly smaller than \eqref{eq:defsmeta2}. Thus, we need fewer evaluations of $f_F$ and the computational cost is reduced.  
\end{remark}

\subsection{Computational cost of the multirate ROCK2 method}\label{sec:compcost}

Given the spectral radii \(\rho_f\) and \(\rho_s\) of the Jacobians of \(\ff\) and \(\fs\), respectively, we now analyze the theoretical speed-up of the mROCK2 method compared to the standard ROCK2 method. For this purpose, we assume \(\varepsilon = 0\) and allow the parameters \(s\) and \(m\) to vary continuously in \(\mathbb{R}\). 

We introduce \(\costf\) and \(\costs\) as relative costs of evaluating \(\ff\) and \(\fs\), normalized with respect to the cost of evaluating the function \(f\). These coefficients satisfy \(\costf, \costs \in [0,1]\) and \(\costf + \costs = 1\). We also suppose that the spectral radius of the Jacobian $\rho=\rho_F+\rho_S$, to allow $\rho_F<\rho_S$.

Since the ROCK2 scheme requires $s=\sqrt{5\tau\rho/4}$ evaluations of $f$ per time step, its cost per time step is
\begin{equation}\label{eq:costrkc}
C_{ROCK2}=s(\costf+\costs)=\sqrt{\frac{5\tau(\rhof+\rhos)}{4}}.
\end{equation}
For the mROCK2 method, on the other hand, we infer
from \eqref{eq:defsmeta} with the approximation $\tilde \beta_s=\frac{4}{5}$,$\beta_m=2$ and $\gamma_\epsilon=\frac{4}{3}$ and $ \frac{m-1}{m^2}\approx 1 $. Thus, we have the system of equations:
\begin{align*}
\frac{4}{3}\,\tau\rho_S &= \frac{4}{5} s^2, &
\eta\,\rho_F &= 2 m^2, &
\text{with} \quad \eta &= \frac{15\tau}{2 s^2} .
\end{align*}

Its gives $s=\sqrt{\frac{5}{3}\tau\rhos}$ external stages and $\eta=\frac{9}{2}\frac{1}{\rho_S} $that it needs $m=\sqrt{\eta \rho_F/2}=\sqrt{\frac{9}{4}\rhof/\rhos}$ internal stages. Since mROCK2 needs $s$ evaluations of $\fs$ and $2\,s\,m$ evaluations of $\ff$, its cost per time step is
\begin{equation}\label{eq:costmrkc}
C_{mROCK2}=s\,\costs+2s\,m\,\costf \approx (1-\costf)\sqrt{\frac{5\tau\rhos}{3}}+\costf\left(\sqrt{15\tau\rhof}\right)
\end{equation}
The ratio between equations \eqref{eq:costrkc} and \eqref{eq:costmrkc} yields the \textit{relative speed-up}
\begin{equation*}\label{eq:speedup}
S=\frac{C_{ROCK2}}{C_{mROCK2}}= \frac{\sqrt{\rhof+\rhos}}{(1-\costf)\sqrt{\frac{4\tau\rhos}{3}}+\costf\left(\sqrt{12\tau\rhof}\right)}
=\sqrt{{\frac{3}{4}}}\frac{\sqrt{1+r_\rho}}{1+\costf\left(\sqrt{ 9r_\rho}-1\right)},
\end{equation*}
with {\it stiffness ratio} $r_\rho=\rhof/\rhos\in [0,\infty)$.

In \cref{fig:speed_up_rhor}, we display the speed-up \(S\) as a function of \(\costf\) for various values of the ratio \(r_\rho = \rho_f / \rho_s\). When \(\costf\) is sufficiently small, the mROCK2 scheme consistently outperforms the standard ROCK2 method (\(S > 1\)). However, for values around \(\costf \approx 0.3\), the mROCK2 scheme becomes slightly less efficient (\(S < 1\)). This case is less important, since we assume that \(\ff\) is relatively inexpensive to evaluate.
\begin{figure}[tbp]
    \centering
    \begin{subfigure}[t]{0.48\textwidth}
        \centering
        \includegraphics[width=\textwidth]{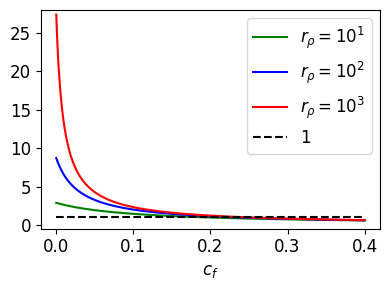}
        \caption{Theoretical speed-up $S$ of the mROCK2 over ROCK2 scheme scheme with respect to $c_F$ and $r_\rho=\rho_F/\rho_S$}
        \label{fig:speed_up_rhor}
    \end{subfigure}
    \hspace{0.02\textwidth}
      \begin{subfigure}[t]{0.48\textwidth}
        \centering
        \includegraphics[width=\textwidth]{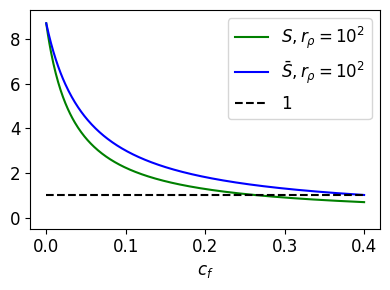}
        \caption{Comparison between $S$ (the relative speed-up without scale separation) and $\bar S$ (the relative speed-up with scale separation) for $r_\rho=100$}
        \label{fig:SvsSbar}
    \end{subfigure}
    \caption{Computational cost of the mROCK2 compared to the ROCK2 method}
\end{figure}

\begin{remark}(Relaxed stability condition)
In the case of scale separation \eqref{eq:defsmeta3}, where the stability analysis is provided in Remark~\ref{remarkscale}, the value of $\eta$ is smaller than in \eqref{eq:defsmeta2}. This reduction leads to fewer evaluations of $f_F$. We denote $\bar{S}$ the relative speed-up of mROCK2 compared to the ROCK2 scheme under condition \eqref{eq:defsmeta3}. The corresponding speed-up is then given by:
 \begin{equation}
\bar{S}=\sqrt{{\frac{3}{4}}}\frac{\sqrt{1+r_\rho}}{1+\costf\left(\sqrt{ 4 r_\rho}-1\right)}.
 \end{equation}
In Figure~\ref{fig:SvsSbar}, we compare it with the speed-up obtained without scale separation, $S$. We have always $\bar S> S$. 
\end{remark}

\subsection{Stability analysis of the mROCK2 method}\label{sec:staban}
We consider the multirate test equation \eqref{eq:mtesteq} and examine both the discrete analogue of the second-order force \eqref{def:fe2} and the modified equation \eqref{eq:odemod2}. The objective of this section is to establish the stability of the mROCK2 method under the condition \eqref{eq:defsmeta}. Let $\Pe_{s,m}(\lambda,\zeta,\tau,\eta)$, denote the stability function of the mROCK2 method, defined by
\begin{equation}\label{eq:defRsm}    
y_{k+1}=\Pe_{s,m}(\lambda,\zeta,\tau,\eta)y_k,
\end{equation}
when we apply the mROCK2 method on the multirate test equation: $\dot{y}=(\lambda+\zeta)y$, $y(0)=y_k$
With this formulation, we are now in a position to state the following theorem.

\begin{theorem}\label{thm:stab_mROCK2}
 Let $\lambda\leq 0$ and $\zeta<0$, $\alpha_m=P_m''(0)$ and $\tilde \beta_s$ the stability coefficient of the ROCK2 method. Then, for all $\tau>0,s,m$ and $\eta$ such that
\begin{align}\label{eq:defsmetascalar}
(1+\alpha_m)\tau |\zeta| &\leq\tilde \beta_ss^2, & \eta|\lambda| &\leq \ell^{\varepsilon_m}_m &\text{with}  && \eta \geq \frac{2\tau(1+\alpha_m)}{\alpha_m\ell_s},
\end{align}
$|\Pe_{s,m}(\lambda,\zeta,\tau,\eta)|\leq 1$, i.e. the mROCK2 scheme is stable.
\end{theorem}
First we have to find the value of $\Pe_{s,m}(\lambda,\zeta,\tau,\eta)$:
\begin{lemma}\label{lemma:closedbue}
	Let $\lambda,\zeta\leq 0$, $\fs(y)=\zeta y$, $\ff(y)=\lambda y$, $\eta>0$, $m,s\in\Nb$ and $y\in\Rb$. Then, the stability function of the mROCK2 method is given by
	\begin{equation*}\label{eq:numutesteq}
\Pe_{s,m}(\lambda,\zeta,\tau,\eta)=\Pe_s\left(\tau\Ps_m(\eta\lambda)(\lambda+\zeta)\left(1-\frac{\eta\lambda\alpha_m}{2}\Ps_m(\eta\lambda)\right)\right)
\end{equation*} 
where $\Ps_m(z) = \frac{P_m(z)-1}{z} $ and $\alpha_m=P_m''(0)$ and 
 $R_s(z)$ is the stability polynomial of the $s$-stage ROCK2 scheme.
	
\end{lemma}
\begin{proof}
    
We use the fact $\bar{f}_{\eta,1} = \Ps_m(\eta\lambda)(\lambda+\zeta)y$ from \eqref{eq:feta1mrkc}. 
Then, by definition of $v_\eta$,
\begin{equation*}
v_\eta=P_m(\eta\lambda)y+\eta\Ps_m(\eta\lambda)(\zeta y-\frac{\eta\lambda\alpha_m}{2}\bar{f}_{\eta,1})
\end{equation*}
Then, we calculate $\bar{f}_{\eta,2}=\frac{1}{\eta}(v_\eta-y)=\Ps_m(\eta\lambda)(\lambda+\zeta)\left(1-\frac{\eta\lambda\alpha_m}{2}\Ps_m(\eta\lambda)\right)$.\newline
Finally, we compute the internal stage into the external stage, which leads to 
\begin{equation*}
y_{n+1} = \Pe_s\left(\tau\Ps_m(\eta\lambda)(\lambda+\zeta)\left(1-\frac{\eta\lambda\alpha_m}{2}\Ps_m(\eta\lambda)\right)\right)y_n,
\end{equation*} 
with $R_s(z)$ the stability polynomial of the $s$-stage ROCK2 scheme. 
\end{proof}
To prove the stability of the mROCK2 scheme, we need the following lemma on the stability property of $\Phi_m(z)$ used and proved in \cite{AbG22}:
\begin{lemma}\label{lemma:disc_stab}
	(see \cite[Lemma 4.3]{AbG22}) Let $m\in\Nb$ and $w\leq 0$. There exists $\overline\varepsilon_m>0$ such that for $\varepsilon\leq\overline\varepsilon_m$, $\Ps_m(z)(z+w)\in [w,0]$ for all $z\in [-\ell_m^\varepsilon,0]$ if and only if, $\Ps_m'(0)|w|\geq 1$, i.e.  $|w|\geq 2/P_m''(0)$ since $\Ps_m'(0)=P_m''(0)/2$. 
\end{lemma}
\begin{proof}[Proof  of Theorem \ref{thm:stab_mROCK2} ]
$|\Pe_{s,m}(\lambda,\zeta,\tau,\eta)|=|\Pe_s(\tau \Ps_m(\eta\lambda)(\lambda+\zeta))|\leq 1$, if  $\tau \Ps_m(\eta\lambda)(\lambda+\zeta)(1-\frac{\alpha_m\eta\lambda}{2}\Ps(\eta \lambda))\in [-\ell_s^\varepsilon,0]$. So we have to prove this condition.
First, we can remark
\begin{align*}
    1-\frac{\alpha_m\eta\lambda}{2}\Ps(\eta \lambda)=1-\frac{\alpha_m}{2}(P_m(\eta\lambda)-1)\in[1,1+\alpha_m]
\end{align*}
because $P_m(\eta\lambda)\in [-1,1]$.
Hence, it is sufficient to prove the equivalent condition:
\begin{align*}
\Phi_m(\eta\lambda)(\eta\lambda+\eta\zeta)&\in [w(\eta),0], & \mbox{with} && w(\eta)&=-\frac{\eta}{(1+\alpha_m)\tau}\ell_s^\varepsilon.
\end{align*}
Since $\eta\lambda\in [-\ell_m^{\varepsilon_m},0]$, it holds $|P_m(\eta\lambda)|\leq 1$ and we deduce that $\Ps_m(\eta\lambda)\geq 0$. Furthermore, \eqref{eq:defsmetascalar} yields $\eta\zeta\geq w(\eta)$ which implies
\begin{equation*}
0\geq \Phi_m(\eta\lambda)(\eta\lambda+\eta\zeta)\geq  \Phi_m(z(\eta))(z(\eta)+w(\eta)),
\end{equation*}
with $z(\eta)=\eta\lambda$. Hence, it is sufficient to show that $\Phi_m(z(\eta))(z(\eta)+w(\eta))\in [w(\eta),0]$ for all $z(\eta)\in [-\ell^{\varepsilon_m}_m,0]$. From \cref{lemma:disc_stab}, we know that
\begin{equation*}
|w(\eta)| \geq  \frac{2}{\alpha_m}
\end{equation*}
is a necessary and sufficient condition which is verified when $\eta \geq \frac{2\tau(1+\alpha_m)}{\alpha_m\ell_s}$.
\end{proof}
\begin{figure}[tbp]
    \centering
    \begin{subfigure}{0.9\textwidth}
        \centering
        \includegraphics[width=\textwidth]{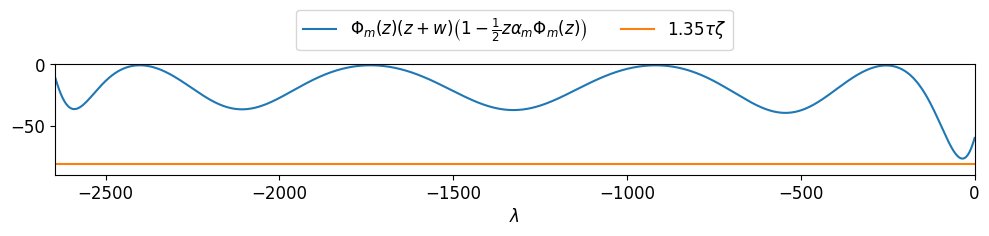}
        \vspace{-0.5cm}
        \caption{The inner stability function: $\tau \Phi_m(z)(z+w)(1-\frac{z}{2}\alpha_m\Phi_m(z))$ for $z=\eta\lambda $, $w=\eta\zeta$ and the stability condition $\ell_s=\gamma_\varepsilon\eta\zeta$.}
        \label{fig:stabilitemrock2a}
    \end{subfigure}
    \begin{subfigure}{0.9\textwidth}
        \centering
        \includegraphics[width=\textwidth]{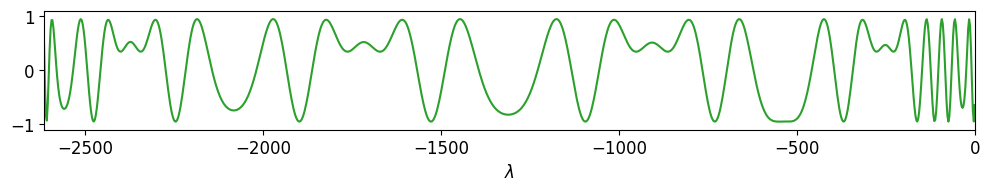}
        \vspace{-0.5cm}
        \caption{Stability function $R_{s,m}(\lambda, \zeta, \tau, \eta)$}
        \label{fig:stabilitemrock2b}
    \end{subfigure}
    
    \caption{The inner and outer stability function of the  mROCK2  for $w=\eta\zeta$, $z=\eta\lambda$, $\tau=1$,  $\eta=\frac{2\tau(1+\alpha_m)}{\alpha_m\ell_s}$, $s=m=10$ and $w \approx 60$ chosen such that the condition $\gamma_\varepsilon\zeta \tau=0.81\ell_s$ for $\varepsilon=0.05$.}
    \label{fig:stabilitemrock2}
\end{figure}
Figure \ref{fig:stabilitemrock2a} represents the inner stability polynomial: $\tau \Phi_m(z)(z+w)(1-\frac{1}{2}z\alpha_m\Phi_m(z))$ for $w=\eta\zeta$, $z=\eta\lambda$, $\tau=1$,  $\eta=\frac{2\tau(1+\alpha_m)}{\alpha_m\ell_s}$, $s=m=10$ and $w \approx 60$ chosen such that the condition $\gamma_\varepsilon\zeta \tau=0.81\ell_s$. The orange line represents the stability condition for the macro-step $\ell_s\ge 1.35\tau\rho_s$. We observe that the method remains stable with a certain margin. Thus, the stability function \(R_{s,m}(\lambda, \zeta, \tau, \eta)\) shown in Figure~\ref{fig:stabilitemrock2b} remains stable for the same parameter values.

 As emphasized in \cite{AbG22} for mRKC, the restriction $\varepsilon\leq \overline\varepsilon_m$ is necessary to prove \cref{lemma:disc_stab}, but probably not necessary in practice also for mROCK2. Indeed, we can numerically verify that for any $\varepsilon \geq 0$, $\Ps_m(z)(z+w)\in [w,0]$ for all $z\in [-\ell_m^\varepsilon,0]$, if and only if, $|w|\geq 2/\Pi_m''(0)$. Hence, we can suppose $\varepsilon_m=\varepsilon$ in \eqref{eq:defsmetascalar} and replace $\ell_m^{\varepsilon_m}$ by $\beta_m m^2$ where $\beta_m$ is the stability coefficient of the RKC method.

Similarly, the condition $\eta \geq \frac{2\tau(1+\alpha_m)}{\alpha_m\ell_s}$ is not restrictive in practice and smaller values of 
\(\eta\) may be chosen. Indeed, the minimum of $\Phi_m(z)(z+w)\left(1 - \tfrac{1}{2} z \alpha_m \Phi_m(z)\right)$ does not correspond to the product of the minimum of $\Phi_m(z)(z+w)$ and the maximum of $\Phi_m(z)(z+w)\left(1 - \tfrac{1}{2} z \alpha_m \Phi_m(z)\right)$. Based on numerical stability experiments for 
\(m \in [1,200]\), we observe that the micro-step size of the 
mRKC method remains applicable, namely $\eta=\frac{6\tau}{\ell_s}\frac{m^2-1}{m^2}$.

\begin{remark}[Stability in the scale separation case]\label{remarkscale}
From \cite{RdS20}, it is shown that 
$
\Phi_m(z)(z+w) \in [0,w],  \text{ for } w > 2 \text{ whenever } z < -8 
\quad ( \lambda < -\tfrac{8}{\eta}).
$
This improves upon the condition 
$
w > \tfrac{2}{P_m''(0)} \approx 6,
$
which holds under a larger damping parameter $\bar \varepsilon = 0.1$. 
In our setting, we keep the original damping $\bar \varepsilon = 0.05$, 
but adjust $\bar \eta$ in order to increase the stability margin, choosing
$
\bar \eta = \frac{2.8\tau}{\ell_s} > \frac{2 \gamma_\varepsilon \tau}{\ell_s}.
$
Figure~\ref{fig:placeholder} illustrates this stability:
the blue curve shows the quantity 
$
\tau \Phi_m(z)(z+w)\left(1 - \tfrac{1}{2} z \alpha_m \Phi_m(z)\right),
$ 
the orange line corresponds to the stability condition needed to apply the macro-step mROCK2, 
and the red vertical line marks the threshold value $\lambda < -\tfrac{8}{\eta}$ for scale separation.

\end{remark}
\begin{figure}[tbp]
    \centering
    \includegraphics[width=0.8\linewidth]{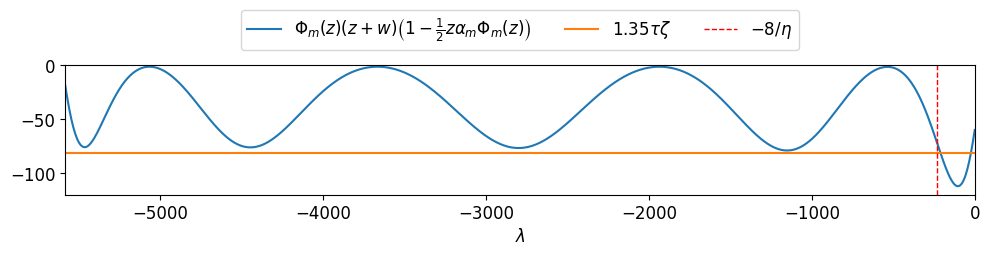}
    \caption{The inner stability function in the case of scale separation:
    $\tau \Phi_m(z)(z+w)(1-\frac{1}{2}z\alpha_m\Phi_m(z))$ for $w=\eta\zeta$, $z=\eta\lambda$, $\tau=1$,  $\frac{2.8\tau}{\ell_s}$, $s=m=10$ and $w \approx 60$ chosen such that the condition $\gamma_\varepsilon\zeta \tau=0.81\ell_s$}
    \label{fig:placeholder}
\end{figure}
\subsection{Convergence analysis of the mROCK2 method}
This section presents the proof of second-order accuracy for mROCK2, the first high-order stabilized multirate method.
\begin{theorem}
    The mROCK2 scheme is second-order accurate
\end{theorem}
\begin{proof}
We estimate the local error after one step. From \cref{def:fe} with $y$ replaced by $y_0$ in 
\cref{eq:defv}.

Let $\bar y_\eta(\tau)$ be the solution of $\bar y_\eta' = \bar{f}_{\eta,2}(\bar y_\eta)$ with $\bar y_\eta(0)=y_0$ and $\bfe$ as in \cref{mROCK2def}. Since the RKC scheme  is first-order accurate then $u_\eta=y_0+\eta f(y_0)+\bigo{\eta^2}$ and $\bar{f}_{\eta,1}(y_0)=f(y_0)+\bigo{\eta}$. Then,
\begin{align*}
    v_\eta=y_0+\eta \left( f(y_0)-\frac{\eta\, \alpha_m}{2}f'(y_0)\bar{f}_{\eta,1}(y_0)\right)+\frac{\eta^2\alpha_m}{2}f'(y_0)f(y_0)+\bigo{\eta^3}=y_0+\eta f(y_0)+\bigo{\eta^3}-
\end{align*}
Thus $\bar{f}_{\eta,2}(y_0)=f(y_0)+\bigo{\eta^2}$ and $f_{\eta,2}'(y_0)=f'(y_0)+\bigo{\eta}$, which yields with $y(\tau)$ the exact solution at time $\tau$:
\begin{align*}
	y(\tau) -\bar y_\eta(\tau) &= \tau(f(y_0)-\bar{f}_{\eta,2}(y_0))+\frac{\tau^2}{2}(f'(y_0)f(y_0)-\bar{f}_{\eta,2}'(y_0)\bar{f}_{\eta,2}(y_0))+\bigo{\tau^3} \\&= \bigo{\eta^2\tau+\eta\tau^2+\tau^3}.
\end{align*}
Finally, let $y_1$ be the solution after one step of the mROCK2 scheme \cref{mROCK2def}, which can also be seen as the solution after one step of the ROCK2 scheme applied to $\bar y_\eta' = \bfe(\bar y_\eta)$. Using the fact that the ROCK2 scheme  has second-order accuracy then
\begin{equation*}
	\bar y_\eta(\tau) -y_1 =  \bigo{\tau^3}.
\end{equation*}
By triangular inequality we obtain $|y(\tau)-y_1|= \bigo{\eta^2\tau+\eta\tau^2+\tau^3}$ and from \cref{eq:defsmeta} follows $\eta\leq \tau$, thus $|y(\tau)-y_1|= \bigo{\tau^3}$ and the scheme is second-order accurate.
\end{proof}
\section{Numerical experiments}\label{sec:numex}

In this section, we compare the performance of the mROCK2 method with the classical ROCK2 method through a series of numerical experiments. We first apply both methods to a stiff nonlinear ODE. Additionally, we include a comparison with  the order one mRKC to illustrate the advantage of the higher order two of accuracy. 

For the second and third experiments, involving PDEs, we study the performance of the methods as both the mesh size $H$ and the time step $\tau$ are refined. The efficiency of mROCK2 is evaluated relatively to ROCK2 in each case.

All simulations are implemented in \texttt{Python} using the \texttt{NumPy} and \texttt{Matplotlib} libraries.

\subsection{Robertson's stiff test problem}
\begin{figure}[tbp]
    \centering
    \begin{subfigure}[t]{0.38\textwidth}
        \centering
        \includegraphics[width=\textwidth]{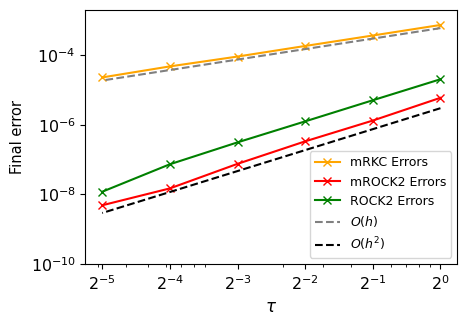}
        \caption{Convergence of mROCK2, mRKC and ROCK2 on Robertson's stiff problem for $T=100$}
        \label{fig:robconv}
    \end{subfigure}
    \hspace{0.02\textwidth}
    \begin{subfigure}[t]{0.45\textwidth}
        \centering
        \includegraphics[width=\textwidth]{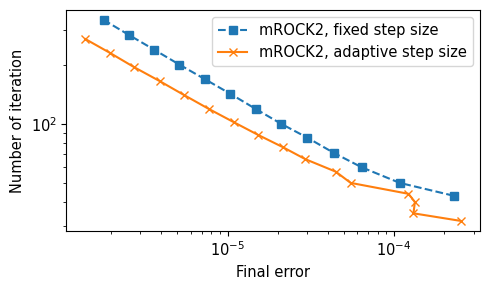}
        \caption{Error versus number of iterations for fixed and adaptive step size strategies at $T=100$}
        \label{fig:step size}
    \end{subfigure}
    \caption{Test of the mROCK2 method on the Robertson's stiff problem }

\end{figure}
First, we study the convergence of the mROCK2 scheme on a popular stiff test problem, Robertson's nonlinear chemical reaction model \cite{HW96}:
\begin{equation}\label{eq:robertson}
\begin{aligned}
y_1' =& -0.04\, y_1+10^4\, y_2\, y_3, & y_1(0)=&1, \\
y_2' =& \,0.04\, y_1-10^4\, y_2\, y_3-3\cdot 10^7\, y_2^2,\qquad \qquad & y_2(0)=&2\cdot 10^{-5},\\
y_3' =& \,3\cdot 10^7\, y_2^2 , & y_3(0)=& 10^{-1},
\end{aligned}
\end{equation}
where $t\in [0,100]$. With this set of parameters and initial conditions, following \cite{AbG22}, the term inducing the most severe stiffness is $-10^4\,y_2\,y_3$. Thus, we let
\begin{align}
\ff(y)=& \begin{pmatrix}
0 \\ -10^4\, y_2\, y_3\\ 0
\end{pmatrix}, &
\fs(y) =& \begin{pmatrix}
-0.04\, y_1+10^4\, y_2\, y_3\\ 0.04\, y_1-3\cdot 10^7\, y_2^2 \\ 3\cdot 10^7\, y_2^2
\end{pmatrix}, &
f(y)=&\ff(y)+\fs(y).
\end{align}
Now, we solve \eqref{eq:robertson} either with the ROCK2, mROCK2  or mRKC scheme using step sizes $\tau = 1/2^k$, $k=0,\ldots,5$. For comparison, we use a reference solution obtained with the standard fourth-order Runge--Kutta scheme using $\tau=10^{-4}$. In \cref{fig:robconv}, we observe that both the mRKC method achieves first-order convergence while mROCK2 and ROCK2 are two to four orders of magnitude more precise thanks to their second-order convergence. Figure \ref{fig:step size} illustrates the efficiency of the adaptive step size strategy for the mROCK2 method by comparing the computational cost versus the error. We contrast the performance of mROCK2 with and without adaptivity, using the same adaptive time step strategy (see Section \ref{sec:step size}).

\subsection{Heat equation on the unit square}\label{sec:exp_conv2}

\begin{figure}[tbp!]
    \centering
    \begin{subfigure}{0.3\textwidth}
        \centering
        \includegraphics[width=\textwidth]{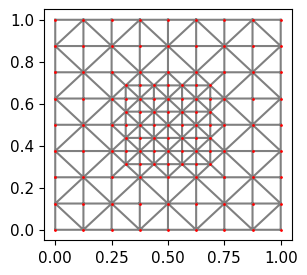}
        \caption{Refined cells for $j=3$, $\rho_s=1.4\cdot10^3$ and $\rho_F=8.5\cdot 10^3$.}
        \label{fig:heat_eq_mROCK2}
    \end{subfigure}
    ~
    \begin{subfigure}{0.45\textwidth}
        \centering
        \includegraphics[width=\textwidth]{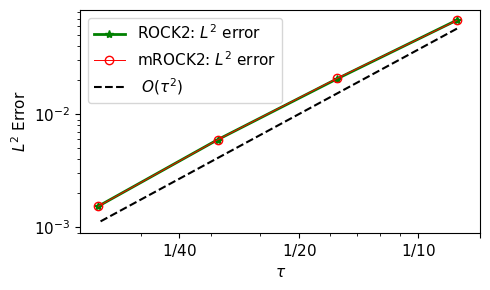}
        \caption{Comparison of the error with the exact solution for mROCK2 and ROCK2.}
        \label{fig:heat_eq_comparison_error}
    \end{subfigure}
    \centering
    \begin{subfigure}{0.4\textwidth}
        \centering
        \includegraphics[width=\textwidth]{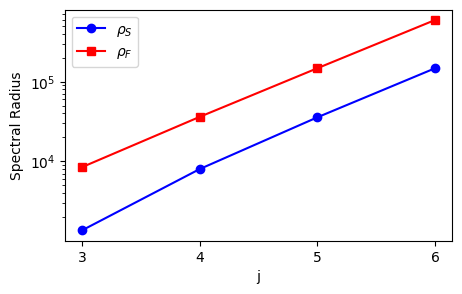}
        \caption{Spectral radius of $A_F$ and $A_S$ as the spatial mesh is refined.}
        \label{fig:spectralradiussimp}
    \end{subfigure}
    ~
    \begin{subfigure}{0.4\textwidth}
        \centering
        \includegraphics[width=\textwidth]{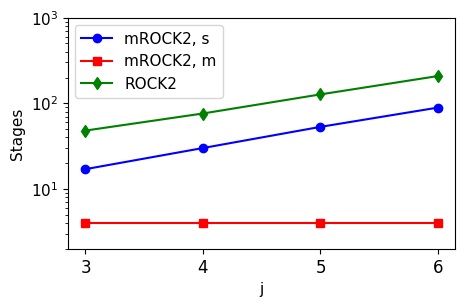}
        \caption{Theoretical numbers of stages $m$ and $s$ for mROCK2, and $n$ for ROCK2.}
        \label{fig:stagesimpl}
    \end{subfigure}

    \caption{Heat equation on the unit square with local mesh refinement.}
    \label{fig:heat_eq_all}
\end{figure}
Next, we verify the space-time convergence properties of the considered multirate methods. To do so, we consider the heat equation in the unit square $\Omega=(0,1)\times (0,1)$,
\begin{equation}\label{eq:par1}
\begin{aligned}
\dt u-\Delta u &=g \qquad &&\text{in }\Omega\times (0,T)\\
u&=0 &&\text{on } \partial\Omega \times (0,T),\\
u&=0 &&\text{on } \Omega\times\{0\},
\end{aligned}
\end{equation}
where $T=1/2$ and $g$ is chosen such that $u(\bx,t)=\sin(\pi x_1)^2\sin(\pi x_2)^2\sin(\pi t)^2$ is the exact solution. 

Starting from a uniform mesh of $2^j\times 2^j$ simplicial elements with $j=2,\ldots,5$, we perform two successive local refinements of all elements located within the square domain \( \Omega_F = (1/4, 3/4) \times (1/4, 3/4) \). Each refinement step is carried out by subdividing interior triangles into four smaller triangles and boundary triangles into two, by inserting a point at the midpoint of the hypotenuse, see top left picture in Figure \ref{fig:the lshape}.

Let $\mathcal{M}$ be the set of elements in the mesh and $\mathcal{M}_F=\{T\in\mathcal{M}\,:\, \overline{T}\cap \overline{\Omega}_F\neq \emptyset\}$ the set of refined elements or their direct neighbors. Then $h=H/4$ is the diameter of the elements inside of $\OF$, with $H$ the diameter of the elements outside of $\OF$. 

Next, we discretize \eqref{eq:par1} in space with first-order FEM  on the mesh $\mathcal{M}$. After inverting the block-diagonal mass matrix ($A=M^{-1}K$ where $M$ is the mass matrix and $K$ is the stiffness matrix), the resulting system is 
\begin{equation}
\dot{y}=A\, y+G,  \qquad\qquad y(0)=y_0,
\end{equation}
where $A\in\Rb^{N\times N}$ and $G\in C([0,T],\Rb^N)$ corresponds to the spatial discretization of $g(\cdot,t)$. Let $D\in\Rb^{N\times N}$ be a diagonal matrix with $D_{ii}=1$ if the $i$th degree of freedom belongs to an element in $\mathcal{M}_F$ and $D_{ii}=0$ otherwise. We also introduce 
\begin{align}\label{eq:splitpar}
A_F=&DA, & A_S=&(I-D)A &\mbox{and} &&\ff(y)=&A_F\, y, & \fs(t,y)=&A_S\, y+G(t),
\end{align}
with $I$ the identity. We recall that the spectral radii $\rhos$ and $\rhof$ of $A_S$ and $A_F$ behave as $\bigo{1/H^2}$ and $\bigo{1/h^2}=\bigo{4/H^2}$, respectively. 

We now consider a sequence of meshes with $j=2,\ldots,5$ and solve \eqref{eq:par1} 
the mRKC  using the same step size $\tau=1/2^{j/2}$ to have $H=O(\tau^2)$. The parameters $s$ and $m$ for mRKC are chosen according to \eqref{eq:defsmeta}. 

In Figure \ref{fig:heat_eq_comparison_error}, we display the error at the final time with respect to the time step \(\tau\) for both mROCK2 and ROCK2. The results show that the two methods exhibit identical convergence behavior. We also plot the spectral radius for each value of \(j\) in Figure \ref{fig:heat_eq_mROCK2}, and compare the number of iterations required by each method in Figure \ref{fig:stagesimpl}. The number of evaluations of the slow function \(f_S\) is observed to be approximately three times higher for the ROCK2 method compared to mROCK2.

\subsection{The L-shape diffusion problem}
Next, we verify the space-time convergence properties on a non-convex surface, which naturally gives rise to singularities near the corners, as discussed in \cite{GMS22}. To do so, we consider a diffusion in the L-shape includes unit square $\Omega=[0,1]\times [0,1]$ without the square $\Omega=[0.5,1]\times [0,0.5]$,
\begin{equation}\label{eq:par2}
\begin{aligned}
\dt u-\Delta u &=g \qquad &&\text{in }\Omega\times [0,T],\\
u&=0 &&\text{in } \partial\Omega \times [0,T],\\
u&=0 &&\text{in } \Omega\times\{0\},
\end{aligned}
\end{equation}
We consider the problem with final time $T = 1$ and a Gaussian source term given by 
\[
g(x,t) = 10 \exp\left(-\frac{(x_1 - x_0)^2 + (x_2 - x_0)^2}{\sigma^2}\right).
\]

As a reference solution, we first perform a standard finite element discretization using a uniform mesh composed of small triangles with diameter $0.005$. The classical finite element method (FEM) is used for the spatial discretization, and the ROCK2 method is applied for time integration with a time step $\tau = 0.005$ to obtain a high-quality approximation of the exact solution. Next, we consider a multiresolution approach: we construct a coarser mesh with larger triangles of diameter $0.1$, except in the region surrounding the singularity, where we refine the mesh with smaller triangles of diameter $0.01$. We then use the matrices $A_F$ and $A_S$ with the same construction  as in the previous experiment (but this time without the square $\Omega=[0.5,1]\times [0,0.5]$), and apply the mROCK2 method for time integration on this refined mesh. Figure \ref{fig:the lshape} shows the approximate solution computed on the coarser mesh with local refinement using the mROCK2 method, alongside a comparison with the solution obtained on a uniformly coarse mesh without refinement. We observe that the locally refined mesh yields better accuracy and efficiency. In cases like this, the mROCK2 method proves to be particularly effective: for example, the number of evaluations at each step of the slow component is reduced to 28 with mROCK2, compared to 96 with the ROCK2 method. We add the plot of the number of iteration of $f_S$ (Figure \ref{costfs}) and $f_F$ (Figure \ref{costff}) and the relative number of iteration for $f$ (Figure \ref{costf}) and compared it with the cost for the ROCK2 method for $H=1/5,1/10,1/20,1/40,1/80$ and $h=H/10$ . The relative cost of evaluating \( f \) is modeled as \( 1 = c_F + c_S \), and the relative number of function evaluations per step in the mROCK2 method is given by \( (1 - c_F)s + 2 c_F s m \). In our experiments, we estimate \( c_F = 0.05 \).

\begin{figure}[tbp]
    \centering
    \hspace{4em}
    \includegraphics[width=0.8\linewidth]{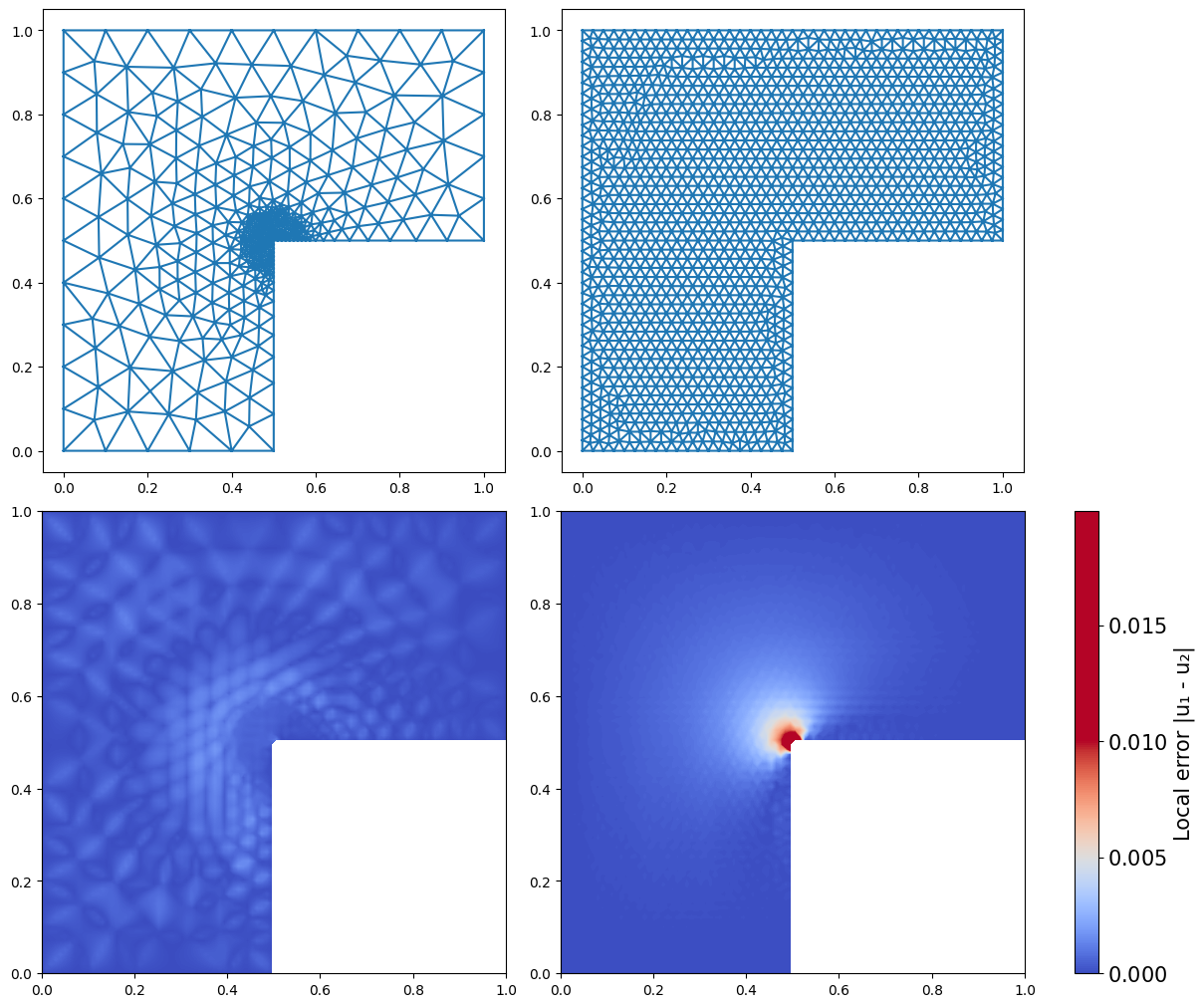}
    \caption{Local numerical error on the L-shaped domain with finite element discretization with a local mesh refinement (left pictures) and a uniform mesh (right pictures).}
    \label{fig:the lshape}
\end{figure}
\begin{figure}[tbp]
    \centering
    
    \begin{subfigure}{0.35\textwidth}
        \centering
        \includegraphics[width=\textwidth]{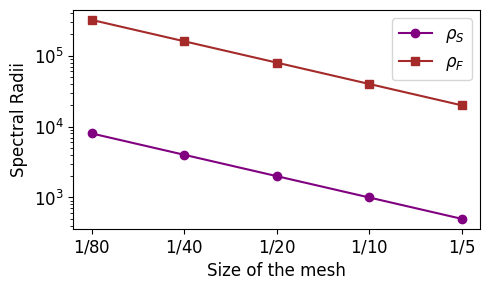}
        \caption{Spectral radii  $\rho_F$ and $rho_s$ of the L-shape problem on the refined mesh with $h=H/10$}
        \label{fig:spectral}
    \end{subfigure}
    \hspace{0.02\textwidth}
    \begin{subfigure}{0.35\textwidth}
        \centering
        \includegraphics[width=\textwidth]{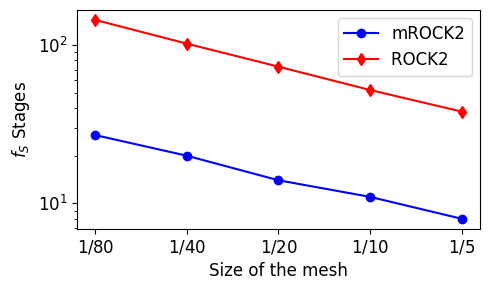}
        \caption{Number of internal evaluations of the slow component $f_S$ required by mROCK2 and ROCK2}
        \label{costfs}
    \end{subfigure}
    
    \vspace{0.5cm} 
    
    \begin{subfigure}{0.35\textwidth}
        \centering
        \includegraphics[width=\textwidth]{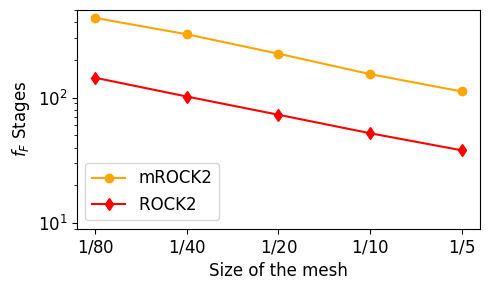}
        \caption{Number of internal evaluations of the fast component $f_F$ required by mROCK2 and ROCK2}
        \label{costff}
    \end{subfigure}
    \hspace{0.02\textwidth}
    \begin{subfigure}{0.35\textwidth}
        \centering
        \includegraphics[width=\textwidth]{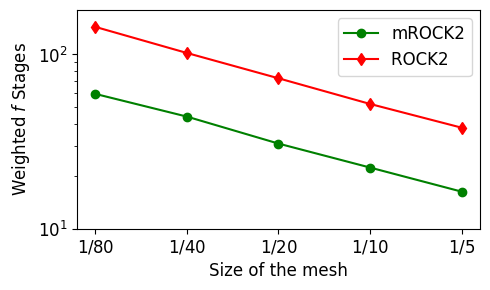}
        \caption{Theoretical number of evaluations of the function $f$ for mROCK2 compared with ROCK2}
        \label{costf}
    \end{subfigure}
    
    \caption{Spectral properties and stage requirements of the mROCK2 and ROCK2 methods for the L-shaped problem}
    \label{lshape_comparison}
\end{figure}

\section*{Acknowledgments}
The authors thank Giacomo Rosilho de Souza for providing his code and for his interest in the further development of multirate explicit stabilized methods. This work was partially supported by the Swiss National Science
Foundation, projects No 200020\_214819 and No. 200020\_192129.

\section{Appendix}




In this appendix, we prove the second inequality in Lemma \ref{thm:optimalbound}. We first establish a few preliminary lemmas:
\begin{lemma}\label{lem:1}
For any $\omega_0 > 1$, the following inequality holds:
\[
\frac{\omega_0}{\sqrt{\omega_0^2 - 1}} \arccosh(\omega_0) > 1.
\]
\end{lemma}

\begin{proof}
Let
$
g(\omega_0) := \omega_0 \arccosh(\omega_0) - \sqrt{\omega_0^2 - 1}, \quad \text{for } \omega_0 > 1.
$

First, we observe that $g(1)=0$. Next,we  compute the derivative $g'(\omega_0) = \arccosh(\omega_0)> 1$. Thus, $g$ is strictly increasing on $[1, \infty)$. Since $g(1)=0$ and $g$ increases thereafter, it follows that $g(\omega_0) > 0$ for all $\omega_0 > 1$. This completes the proof.
\end{proof}
\begin{lemma}\label{lem:2}
     For any integer $m$ and $\varepsilon>0$,
     \[
     m\arccosh\left(1+\frac{\varepsilon}{m^2}\right)<\sqrt{2\varepsilon}
     \]
\end{lemma}
\begin{proof}
    We rewrite the inequality by setting \( \delta = \frac{\varepsilon}{m^2} \). Hence, it remains to prove that $\arccosh(1+\delta)<\sqrt{2\delta}$ for $\delta>0$.
    
    We define the function
    $h(\delta)=\sqrt{2\delta}-\ \arccosh(1+\delta).
    $
    We want to show that \( h(\delta) > 0 \) for all \( \delta > 0 \).
    $h(0)=0$ and $h'(\delta)=\frac{1}{\sqrt{2\delta}}-\frac{1}{\sqrt{2\delta+\delta^2}}>0$ for $\delta>0$. Therefore, \( h \) is strictly increasing on \( (0, \infty) \), and since \( h(0) = 0 \), it follows that \( h(\delta) > 0 \) for all \( \delta > 0 \), which completes the proof.
    \end{proof}
\begin{lemma} \label{lem:3}
For any $x>0$, the function $\frac{1}{\tanh^2 (x)}-\frac{1}{x\tanh(x)}$ is an increasing function    
\end{lemma}
\begin{proof}
    Let $\phi(x)=\frac{1}{\tanh^2 x}-\frac{1}{x\tanh(x)}$. We will show that $\phi'(x)>0$ for $x>0$.
    \begin{align*}
      \phi'(x)&=
      \frac{\cosh(x)\sinh^2(x)+x\sinh(x)-2x^2\cosh(x)}{x^2\sinh^3(x)}.
    \end{align*}    
     Since the denominator \( x^2 \sinh^3(x) > 0 \) for all \( x > 0 \), the sign of \( \phi'(x) \) depends on the numerator. Thus, showing that $\phi'(x)>0$ is the same as:
    \[
    \psi(x)=\cosh(x)\sinh^2(x)+x\sinh(x)-2x^2\cosh(x)>0 \quad \text{for} \quad x>0.
    \]
    We consider the Taylor series expansions of hyperbolic functions. In particular, from the series for $\sinh(x)$, it follows that
$\sinh^2(x) >x^2 + \frac{x^4}{3}$ for all $ x > 0$. Exploiting both the Taylor series and this inequality, we bound \( \psi(x) \) from below:
    \begin{align*}
     \psi(x)&>\frac{x^4}{3}\cosh(x)+x\sinh(x)-x^2\cosh(x)\\
     &=\sum^\infty_{n=0}\frac{1}{3}\frac{x^{2n+4}}{(2n)!}+\frac{x^{2n+2}}{(2n+1)!}-\frac{x^{2n+2}}{(2n)!}\\
     &=x^2-x^2+\sum^\infty_{n=1}\left(\frac{1}{3}\frac{1}{(2n-2)!}+\frac{1}{(2n+1)!}-\frac{1}{(2n)!}\right)x^{2n+2}\\
     &=\frac{1}{3}\sum^\infty_{n=1}\left(\frac{8n^3-6n-2}{(2n+1)!}\right)y^{2n+2}.
    \end{align*} 
    If $n=1$, $8n^3-6n-2=0$ and it's easy to prove that for all $n>1$, $8n^3-6n-2>0$,  so every term in the sum with \( n > 1 \) is positive. Hence, \( \psi(x) > 0 \) for all \( x > 0 \), and consequently \( \phi'(x) > 0 \) for all \( x > 0 \). This proves that \( \phi \) is strictly increasing on \( (0, \infty) \).
\end{proof}
\begin{proof}[Proof of the second bound in Lemma \ref{thm:optimalbound}]
By definition, \( \alpha_m(\varepsilon) \) is given by
\[
\alpha_m(\varepsilon) = \frac{T_m(\omega_0) \, T_m''(\omega_0)}{T_m'(\omega_0)^2}, \qquad \text{where } \omega_0 = 1 + \frac{\varepsilon}{m^2}.
\]
The Chebyshev polynomials of the first kind satisfy for $\omega_0>1$, that $
T_m(\omega_0) = \cosh(m \, \arccosh(\omega_0))$ and their derivatives can be expressed as:
\[
T_m'(\omega_0) = \frac{m \, \sinh(m \, \arccosh(\omega_0))}{\sqrt{\omega_0^2 - 1}}, \qquad 
T_m''(\omega_0) = \frac{m^2 \, \cosh(m \, \arccosh(\omega_0))}{\omega_0^2 - 1} - \frac{m \, \omega_0 \, \sinh(m \, \arccosh(\omega_0))}{(\omega_0^2 - 1)^{3/2}}.
\]
Using these expressions, we obtain the formula:
\[
\alpha_m(\varepsilon) = \frac{1}{\tanh^2(x)} - \frac{1}{m} \cdot \frac{\omega_0}{\sqrt{\omega_0^2 - 1}} \cdot \frac{1}{\tanh(x)}=\frac{1}{\tanh(x)} \left( \frac{1}{\tanh(x)} - \frac{1}{m} \cdot \frac{\omega_0}{\sqrt{\omega_0^2 - 1}} \right),
\]
where we set \( x := m \, \arccosh(\omega_0) \). Thus, $
\frac{1}{m} \cdot \frac{\omega_0}{\sqrt{\omega_0^2 - 1}} = \frac{1}{x} \cdot \frac{\omega_0}{\sqrt{\omega_0^2 - 1}} \cdot \arccosh(\omega_0).
$ Finally we obtain:
\[
\alpha_m(\varepsilon) = \frac{1}{\tanh(x)} \left( \frac{1}{\tanh(x)} - \frac{1}{x} \cdot \frac{\omega_0}{\sqrt{\omega_0^2 - 1}} \cdot \arccosh(\omega_0) \right).
\]
From \cref{lem:1}, we deduce that 
$
\frac{1}{x} \cdot \frac{\omega_0}{\sqrt{\omega_0^2 - 1}} \cdot \arccosh(\omega_0) > \frac{1}{x}.
$ Therefore,
\[
\alpha_m(\varepsilon) < \frac{1}{\tanh(x)} \left( \frac{1}{\tanh(x)} - \frac{1}{x} \right).
\]
Next, since \( x < \sqrt{2\varepsilon} \) by \cref{lem:2}, and since the function $
x \mapsto \frac{1}{\tanh(x)}\left( \frac{1}{\tanh(x)} - \frac{1}{x} \right)
$
is increasing by \cref{lem:3}, we conclude that:
\[
\alpha_m(\varepsilon) < \frac{1}{\tanh(\sqrt{2\varepsilon})} \left( \frac{1}{\tanh(\sqrt{2\varepsilon})} - \frac{1}{\sqrt{2\varepsilon}} \right) \
\]
This completes the proof of the first part of the theorem.

We now prove $lim_{m \rightarrow \infty} \alpha_m(\varepsilon)=\frac{1}{\tanh(\sqrt{2\varepsilon})} \left( \frac{1}{\tanh(\sqrt{2\varepsilon})} - \frac{1}{\sqrt{2\varepsilon}} \right)$. We recall the definition of $\alpha_m$:
\[
\alpha_m(\varepsilon) = \frac{1}{\tanh^2(m \, \arccosh(\omega_0))} - \frac{1}{m} \cdot \frac{\omega_0}{\sqrt{\omega_0^2 - 1}} \cdot \frac{1}{\tanh(m \, \arccosh(\omega_0))},
\]
with \( \omega_0 = 1 + \frac{\varepsilon}{m^2} \).
As \( m \to \infty \), we expand \( \arccosh(\omega_0) \) using 
$
\arccosh(1 + \delta) = \sqrt{2\delta} \left( 1 + \bigo{\delta} \right)$, as  $\delta \to 0^+.$
Applying this with \( \delta = \frac{\varepsilon}{m^2} \), we get
\[
m \, \arccosh\left(1 + \frac{\varepsilon}{m^2}\right) = \sqrt{2\varepsilon} \left( 1 + \bigo{\frac{\varepsilon}{m}} \right).
\]
Hence, as \( m \to \infty \),
\begin{align*}
\tanh(m \, \arccosh(\omega_0)) &\to \tanh(\sqrt{2\varepsilon}) \quad \text{and} \quad
\frac{1}{m\sqrt{\omega_0^2 - 1}} = \frac{1}{m\sqrt{\frac{2\varepsilon}{m^2} + \frac{\varepsilon^2}{m^4}}} \longrightarrow \frac{1}{\sqrt{2\varepsilon}}.
\end{align*}
Taking the limit in the original expression, we conclude:
\[
\lim_{m \to \infty} \alpha_m(\varepsilon)
= \frac{1}{\tanh^2(\sqrt{2\varepsilon})} - \frac{1}{\sqrt{2\varepsilon} \tanh(\sqrt{2\varepsilon})}
= \frac{1}{\tanh(\sqrt{2\varepsilon})} \left( \frac{1}{\tanh(\sqrt{2\varepsilon})} - \frac{1}{\sqrt{2\varepsilon}} \right).
\]
This completes the proof of the second inequality in \eqref{inequalityalpham1}.

\end{proof}

\bibliographystyle{myplain}
{\small \bibliography{sample}}

\end{document}